\documentclass[11pt]{amsart}
\usepackage[top = 1in, left = 1.25in, right = 1.5in, bottom = 1in]{geometry}
\usepackage{amsmath}
\usepackage{amsrefs}
\usepackage{mathrsfs}
\usepackage{verbatim}
\usepackage{graphicx}
\usepackage{tikz}
\usepackage{lineno}

\def\be{\begin{equation}}
	\def\en{\end{equation}}

\def\bee{\begin{equation*}}
	\def\ene{\end{equation*}}

\def\min{\text{min}}
\def\max{\text{max}}

\definecolor{darkgreen}{rgb}{.1,.6,0}

\newcommand{\Z}{\mathbb{Z}}
\newcommand{\N}{\mathbb{N}}

\newtheorem{theorem}{Theorem}[section] 
\newtheorem{lemma}[theorem]{Lemma}     
\newtheorem{corollary}[theorem]{Corollary}
\newtheorem{proposition}[theorem]{Proposition}

\newtheorem{condition}[theorem]{Condition}
\theoremstyle{definition}
\newtheorem{definition}[theorem]{Definition}

\newtheorem*{ack*}{Acknowledgment}
\newtheorem{remark}[theorem]{Remark}

\numberwithin{equation}{section}

\title[Polynomial Shape Adics]{Polynomial shape adic systems are inherently expansive}

\author{Sarah Frick}
\address{Department of Mathematics, Furman University, Greenville, SC 29613 USA}
\email{sarah.frick@furman.edu}
\author{Karl Petersen}
\address{Department of Mathematics,
	CB 3250 Phillips Hall,
	University of North Carolina,
	Chapel Hill, NC 27599 USA}
\email{petersen@math.unc.edu}
\author{Sandi Shields}
\address{College of Charleston, 66 George St., Charleston, SC 29424-0001 USA}
\email{shieldss@cofc.edu}

\date{\today}

	\subjclass[2020]{37B10, 37B02, 28D05}
\keywords{Bratteli-Vershik system, expansiveness}

\begin{document}

	\begin{abstract}
{To study any dynamical system it is useful to find a partition that allows essentially faithful encoding (injective, up to a small exceptional set) into a subshift. 
	Most topological and measure-theoretic systems can be represented by Bratteli-Vershik (or adic, or BV) systems. 
	So it is natural to ask when can a BV system be encoded essentially faithfully. 
	We show here that for BV diagrams defined by homogeneous positive integer multivariable polynomials, and a wide family of their generalizations, which we call polynomial shape diagrams, for every choice of the edge ordering the coding according to initial path segments of a fixed finite length is injective off of a negligible exceptional set.}
	\end{abstract}
\maketitle

	\section{Introduction}\label{sec:intro}
	Measure-preserving and topological dynamical systems have representations as successor (Vershik) maps on path spaces of Bratteli diagrams, \cites{HPS1992, Vershik1981Markov, verliv}.
	Whether they can be coded as subshifts is the question of expansiveness. 
	The combinatorial nature of these symbolic dynamical representations provides {a 
 viewpoint that can suggest new questions and new methods to deal with them.} 
 Previously \cites{xman,MelaPetersen2005} X. M\'{e}la showed that in the famous 
 Pascal adic system, with the left-right ordering, {the orbits of infinite paths 
 from the root that do not eventually follow only minimal or only maximal edges} can 
 be faithfully coded by the partition determined by the first  edge.  
 {Frick \cite{Frick2009}*{Theorem 4.3} extended this kind of expansiveness to {\em limited scope} systems, including those defined by polynomials in one variable (in our current setting, two variables).} 
 We showed 
 \cite{FPS2017} that for {\em every} ordering {of the Pascal diagram}, the set of {paths not in the orbit of a minimal or maximal path} can be 
 faithfully coded by the first three edges, so that this system is {\em inherently 
 expansive}. Our aim here is to extend this result to a fairly wide {family} of diagrams, {which we call {\em polynomial shape},}
 {that includes the diagrams defined by homogeneous multivariable polynomials, the Pascal, Euler, and reverse Euler diagrams in higher dimensions, and more, {generalizing systems previously studied in, for example, \cites{Kerov1989,verliv,MelaPetersen2005, xman,Mela2006, Bailey2006,BKPS2006,FP2008, FP2010, Strasser2011, FrickOrmes2013}.}  
 	Recent papers related to expansiveness 
 of Bratteli-Vershik systems include
	\cites{dm2008,BezuglyiKwiatkowskiMedynets2009,FPS2017,AFP,Berthe2017,FPS2020}.
	
The next section begins by recalling basic definitions and terminology about Bratteli diagrams and Bratteli-Vershik systems. 
	When coding orbits it is necessary to discard those of maximal and minimal paths, and our proofs require paths with dense orbits. 
	Proposition \ref{prop:minmax} establishes that, under mild conditions on the diagram, for each ordering the {\em exceptional set} (one-sided orbits of maximal and minimal paths together with the nondense orbits) is meager (first category). 
	The set of invariant measures depends only on the tail relation, and so is independent of the choice of ordering.
    For the diagrams dealt with in this paper, the exceptional set also has measure zero for every fully supported ergodic invariant probability measure ---see Remark \ref{rem:measures} and
  Proposition \ref{prop:measures}.
 
	In Section \ref{sec:systems} we define the {\em polynomial diagrams} and their generalization {\em polynomial shape diagrams} that are the subject of this paper, 
	 establish terminology and notation for keeping track of connections between vertices, and identify {\em distinguished source vertices} which will be useful as pivots for moving around in the diagram.  
{In the proof of the main theorem it will be essential to observe the intersections of {\emph{source sets}}  
	{(see Definition \ref{def:sourceset})}
	of certain vertices,} so in Section \ref{sec:covered} we determine exactly which vertices at a given level have their source sets inside those of other vertices at that level. 
To prove that orbits of paths not in the exceptional set can be distinguished from one another by observing their initial finite segments of a fixed length, one necessarily proceeds by contradiction. 
In Section \ref{sec:chains}, on the basis of the hypothesis that there exist pairs of paths that cannot be so distinguished, we build machinery (chains and links) within the diagram that connect the vertices through which these paths pass at two adjacent levels. 
Although all of this is fictional, since the 
ultimate aim is to show that none of it is 
possible, in the final Section \ref{sec:main} 
the constructions do produce a contradiction 
to how the pair of paths would have to move 
through the diagram, thereby proving the main 
theorem (Theorem \ref{thm:main}): Every 
polynomial shape diagram is inherently 
expansive.

{Now we summarize with more detail the main steps of the argument.}

{The idea is to show that if a {\em depth $i$ pair} (\ref{def:depth}) did indeed exist, then we could construct a long linearly ordered {\em chain} (\ref{def:chain}) consisting of {\em uncovered} (\ref{def:covered}) {\em shared} and {\em splitting vertices} (\ref{def:chain}).}
We then argue that this leads to a contradiction, as follows:

Assuming that the orbits of both paths are dense, eventually one of them, say $x'$, must be the first to change vertices at level $i+2$. 	
{For $i$ sufficiently large, the constructed chain will be long enough so that {(1)} the first vertex at level $i+1$ that $x'$ passes through after switching vertices {twice} at level $i+2$ is still part of the chain; 
	and (2) if $z$ is the second vertex at level $i+2$ through which the forward orbit of $x'$ passes while in the chain, then there is a vertex in its source set $S(z)$ that is not one of the splitting vertices met by the orbit of $x'$ during the time that it moves through the chain. (This is argued in the proof of the main theorem using Lemma \ref{lem:SourceSet}.)
	Hence, the orbit of $x'$ does not meet every vertex in $S(z)$ before changing vertices again at level $i+2$.  But according to  the definition of the Vershik map, it must. Therefore there can be no depth $i$ pair.

	{To produce the straight chain used in this proof, we first show that there exist a large region of uncovered vertices at level $i+1$}
	{(\ref{lem:LinkLemma}, \ref{prop:covered})} 
	{and a time when the orbits of $x$ and $x'$ meet distinct vertices in this region,  $w_0$ and $w_1$, respectively  (\ref{lem:GettingToTheGivenMV}).}
	The vertices $w_0$ and $w_1$ are then chosen to be the first and second splitting vertices of our chain.
	
	Since  $w_1$ is uncovered, there exists a distinguished direction and \emph{distinguished vertex} (\ref{def:dsv}) {$u$} in $S(w_1)\setminus S(w_0)$.
	To find the next splitting vertex of the chain, we consider the first time (forward or backward) that the orbit of $x'$ meets {$u$}, while still at $w_1$.  
	Since $x$ and $x'$ have the same $i$-coding, the orbit of $x$ must switch vertices at level $i+1$ by this time, while the orbit of $x'$ does not leave $w_1$ until afterwards.  
	The fact that $w_0$ is {in our desired region of uncovered vertices} (in particular, $S(w_0)$ is not contained in $S(w_1)$) is then sufficient (and necessary) {for us} to show that when the orbit of $x$ leaves $w_0$, it cannot switch to $w_1$ (Lemma \ref{lem:LinkLemma}).

	Hence, the next vertex $w_2$ at level $i+1$ met by the orbit of $x$ is distinct from both $w_0$ and $w_1$. 
	We let $w_2$ be the third vertex in our chain.
	{Since the region of uncovered vertices containing $w_0$ and $w_1$ is large enough to contain $w_2$, we can repeat the process, using $w_1$ and $w_2$,  with the roles of $x$ and $x'$ reversed.} 
	{We continue in this manner until we produce a sufficiently long chain whose splitting vertices are met, in an alternating fashion, by the orbits of $x$ and $x'$, respectively.
		The properties of the distinguished source vertices and our region ensure that each new splitting vertex is in our region and distinct from all previous splitting vertices in the chain.} 
	{Thus the chain has the properties necessary to complete the proof as sketched above} {(Theorem \ref{thm:main})}.

	To highlight the essentials of the argument even more, we now very briefly describe how it applies to the two-dimensional Pascal system, which is defined by the polynomial in two variables $p(x_1,x_2) = x_1 + x_2$.
	The proof of the main theorem, when applied to this system, would require $i \geq 20$, but in fact we need only
	$i \geq 3$.
	
	If there were a depth $i$ pair $x$, $x'$, then at some time their orbits would follow different paths into level $i+1$, and
	since there is at most one edge between any two vertices, they would meet different vertices $w_0$ and $w_1$ at level $i+1$ at this time, which
	we can assume is time 0.
	
	Note that each row of the two-dimensional Pascal diagram has only two covered vertices and they are at the ends,
	so every vertex after the root is in the source set of at most one covered vertex.
	Since $x$ and $x'$ meet the same vertex at level $i$, at least one of $w_0$ or $w_1$ is uncovered.
	
	Consider the case where one of these vertices, say  $w_0$, is covered. 
	Replacing $T$ with $T^{-1}$ if necessary, assume that the orbit of $x'$ meets {the (automatically distinguished) vertex} $u \in S(w_1) \setminus S(w_0)$ at some time $m>0$, while still at $w_1$. 
	{Since $m>0$, it is then necessarily the case that the edge the path $x'$ follows into $w_1$ is minimal according to the edge ordering while the edge that $T^mx'$ follows from $u$ to $w_1$ is maximal.} 
	{Then $T^mx$ also meets $u$.}
	Since $u \notin S(w_0)$, the vertex $w_2$ at level $i+1$ {met by $T^mx$} is clearly distinct from $w_0$.
	It also follows that $w_2 \neq w_1$, since $T^mx$ follows a minimal edge into $w_2$ and the minimal edge into $w_1$ has source in $S(w_0)$.	 
	
	Note that $w_1$ shares one of its source vertices with the covered vertex $w_0$ and the other with $w_2$.
	So since $i+1 \geq 3$,  $w_2$ is clearly uncovered. 
	Hence, we can now add to the current chain (whose splitting vertices are $w_0$, $w_1$ and $w_2$) {by using the (distinguished) vertex in $S(w_2) \setminus S(w_1)$ to extend the subchain (link, $w_1$ to $w_2$) to a new vertex, $w_3$.}
	We can continue to extend our chain, {moving in a fixed direction away from $w_0$, at step $k+1$ adding a splitting vertex $w_{k+1}$ whose source set intersects that of $w_k$ and which is distinct from $w_0, \dots, w_k$}, until its last splitting vertex is covered (i.e. until we reach the other edge of the diagram). 
	{The existence of the chain forces a consistent ordering (left to right or right to left) on the edges entering the vertices at level $i+1$.}
	Each of the uncovered vertices at level $i+1 \geq 4$ (of which there are at least three) is a splitting vertex in our chain, and
	our construction ensures that collectively the orbits of $x$ and $x'$ meet them all, in an alternating manner, before one of them meets the last (covered) splitting vertex in the chain.
	We claim that this does not happen until after the orbit of $x$ has changed vertices twice at level $i+2$.

	Looking at the diagram structure, it is easy to see that two splitting vertices can both be in the source set of a vertex $z$ at level $i+2$ only if they are adjacent in the chain
	{(equivalently, next to each other on row $i+1$ of the diagram).} 
	{Because the orbits of $x$ and $x'$ are skipping every other splitting vertex in the chain, every time one of them changes splitting vertices it changes vertices at level $i+2$.}
	Furthermore, the chain has length at least 5 (counting the two covered vertices at its beginning and end), so the orbit of $x$, {beginning at} $w_0$, will meet at least 3 splitting vertices.
	Hence, the orbit of $x$ changes vertices at level $i+2$ at least twice before reaching the end of the chain.

	{If $z$ is the second vertex at level $i+2$ through which the forward orbit of $x$ passes while in the constructed chain, $|S(z)| = 2$.}
	{The orbit of $x$ meets $z$ and a splitting vertex {(at level $i+1$) that} is adjacent to two other splitting vertices in the chain.
		 It does not meet either of these adjacent vertices while moving through the chain, since it skips {alternate splitting vertices} in the chain,	
		However one of these splitting vertices must be in $S(z)$, contradicting the fact that the orbit of $x$ must meet every vertex in $S(z)$ before changing vertices again at level $i+2$. 
		So there can be no depth $i$ pair.}

	If we start near the middle of the diagram, with $w_0$ and $w_1$ both uncovered, 
	we build the chain in both directions, using forward time for one direction and backward for the other (starting with {(distinguished)} vertices in $S(w_1) \setminus S(w_0)$ and $S(w_0) \setminus S(w_1)$ respectively). 
	Replacing $x$ and $x'$ with paths in their orbits that simultaneously meet two splitting vertices of the chain, the first of which is covered, we will be in the previously considered case.

{\section{Background}\label{sec:background}}

We now recall (in fact partly quote) from \cite{PetersenShields2023} 
the standard basic definitions about Bratteli diagrams, ordered 
Bratteli diagrams, and Bratteli-Vershik systems.
	For further terminology and background, see 
 \cites{HPS1992,GPS1995,Durand2010,BK2016,FPS2017} 
 and their references.

	A Bratteli diagram (here also just called a {\em diagram}) is a countably infinite, directed, graded graph. 
		For each $n=0,1,2,\dots$ there is a finite nonempty set of vertices $\mathcal{V}_n$.	$\mathcal V_0$ consists of a single vertex, called the ``root".  
The set of edges is the disjoint union of finite nonempty sets $\mathcal E_n, n\geq 0$, with $\mathcal E_n$ denoting the set of edges with source in $\mathcal V_{n}$ and target in $\mathcal{V}_{n+1}$. 
					The {\em source map} $s: \mathcal E_n \to \mathcal V_{n}$ and {\em target map} $t:\mathcal E_n \to \mathcal{V}_{n+1}$ are defined as usual.
						The {\em source set} of a vertex $w \in \mathcal V_{n+1}$ is $S(w)= s(t^{-1}\{w\})$. 
						(To simplify notation, we will write $t^{-1}\{w\} = t^{-1}w$.) 
						We define $\mathcal V=\cup_n\mathcal{V}_n$, $\mathcal{E}=\cup_n\mathcal{E}_n$ and denote the diagram by $\mathcal B=(\mathcal{V},\mathcal{E})$.
											
Every vertex other than the root has at least one incoming edge and at least one outgoing edge, and 
there can be multiple edges between pairs of vertices. 

{The space $X$ 
is the set of infinite paths (sequences $x=x_0 x_1\dots$, each $x_i$ 
in $\mathcal{E}_{i}$) starting at the root 
at level $i=0$. 
For a path $x$ we denote by $v_i(x)$ the vertex of the 
path at level $i$.
In other words,  $v_i(x)=s(x_i)=t(x_{i-1})$. 
$X$ is a compact metric space when we specify that 
two paths have distance $1/2^n$ if they agree from levels $0$ to $n$ 
and disagree leaving level $n$. We will avoid degenerate situations 
and only consider diagrams for which $X$ is homeomorphic to the Cantor 
set. The cylinder sets $x_{[0,n]}=\{y \in X: y_i = x_i, i=0, \dots, n\}$, 
for $n \geq 0$ and $x \in X$, are clopen sets that generate the topology. }

The edges entering each vertex can be totally ordered
	by 
 specifying a map $\xi: \mathcal{E}\to \mathbb{N}$ such that for $n>0$ and $w\in \mathcal{V}_n$, $\xi$ restricted to $t^{-1}w$
 is a bijective map with range $\{1, 2,...,|t^{-1}w|\}$. The diagram together with such an order is called an {\em ordered Bratteli diagram}.
	
	Two paths $x$ and $y$ are {\em comparable}, or {\em tail equivalent}, if 
   they are cofinal: there is a smallest $N>0$ such that $x_N=y_N$ for all $n \geq 
   N$. 
	In this case $x_{N-1} \neq y_{N-1}$; we agree that $\xi(x)<\xi(y)$ 
 if $\xi(x_{N-1}) < \xi(y_{N-1})$, and $x>y$ if not.
The set of {\em minimal paths}, meaning those all of whose edges are minimal into all 
of their vertices, will be denoted by $X_{\min}$, and similarly the set of {\em 
maximal paths} will be denoted by $X_{\max}$. 
	The Vershik, or adic, map $T=T_\xi$ is defined from the set of nonmaximal paths 
 to the set of nonminimal paths by mapping each path $x$ to its {\em successor}, the 
 smallest $y>x$. 
	The pair $(X,T)$ is called a {\em Bratteli-Vershik system}.

 We do not always need an ordering or transformation in order to use dynamical concepts and terminology. 
	Given a diagram, we think of the tail relation as the ``orbit" relation: two points are in the same ``orbit" if they agree from some level downward. 
	If $x$ and $x'$ are tail related, then for each ordering $\xi$ there is $n \in \Z$ such that $T_\xi ^nx =x'$; thus every tail-relation orbit is a $T_\xi$-orbit.
	
	An {\em invariant set} is one that is a union of orbits, i.e., it is saturated with respect to the tail relation. 
	Note that each orbit is at most countably infinite. 
	An \emph{invariant measure} $\mu$ is one that for each vertex $w$ assigns equal measure to all the cylinder sets determined by paths from the root to $w$: if $v_n(x)=v_n(y)=w$, then $\mu(x_{[0,n]})=\mu(y_{[0,n]})$. 

 We denote by $X'_\xi$ the set consisting of the backward $T_\xi$-orbits of maximal paths and the forward orbits of minimal paths. Thus 
 $X'_\xi$ consists of the paths whose edges are all eventually minimal or all 
 eventually maximal. 
 
 Denote by $X_0$ the set of paths whose orbits are not dense.
 Then $X_0$ has measure $0$ for every fully supported ergodic invariant measure on $X$. 
 Unlike the situation for the two-dimensional Pascal system with left-right order \cite[Proposition 2.3]{MelaPetersen2005}, even for a polynomial shape system the set $X_0$ need not coincide with $X'_\xi$ for any order. 
 See Proposition \ref{prop:orbits} for a precise description of $X_0$ in the case of polynomial shape systems.
 
\begin{definition}\label{def:genericpaths}
	Given an ordering $\xi$ of a Bratteli diagram as above, we denote by $X(\xi)$ the 
 set of paths in the path space $X$ that have a dense orbit and are not in $X'_\xi$. 
 The {\em exceptional set} is $X'_{\xi}\cup X_0=X \setminus X(\xi)$.
\end{definition}
	
We say that two vertices $v,w$ are {\em connected} if there is a finite downward directed path of 
consecutive edges in the diagram that begins at one of the vertices and ends at the 
other. If for every $n$ and every pair of vertices $v_1,v_2 \in \mathcal V_n$ there 
are $m > n$ and a vertex $w \in \mathcal V_m$ that is connected to both $v_1$ and 
$v_2$, we say the diagram is {\em connected}. 

 {We say that a diagram is {\em topologically transitive} if given {nonempty} open sets $U,V$, there is a point (infinite path) in $U$ which is tail equivalent to some point in $V$. 
 	Just as for homemorphisms on compact metric spaces, there are several conditions equivalent to topological transitivity. 
 	{Recall that a set is called {\em meager}, or {\em first category}, if it is a countable union of nowhere dense sets, and the complement of a meager set is called {\em comeager} or {\em residual}.}
  \begin{proposition}\label{prop:toptrans}
 	The following properties of a Bratteli diagram are equivalent:\\
 	(1) The diagram is connected, in the sense just defined.\\
 	(2) The diagram is topologically transitive.\\
 	(3) Every proper closed {invariant set} is nowhere dense.\\
 	(4) There is a point with dense orbit.\\
 	(5) The set of points with dense orbit is comeager. 
 \end{proposition}}
 \begin{proof}
 {The equivalence of (1), (2), and (3) is immediate. 
 	Since $X$ is a Baire space, (5) implies (4), and clearly (4) implies (2). 
 	We show now that (1) implies (5).}
 	
 {Given a finite path $\alpha=\alpha_0 \dots \alpha_n$ starting at the root, denote by $[\alpha]$ the cylinder set consisting of all paths that begin with $\alpha$:
 	\be 
 	[\alpha]=\{x \in X: x_{[0,n]}=\alpha\}.
 	\en
 	{If $x$ is a path whose orbit $\mathcal O(x)$ is not dense}, then there is such a finite path $\alpha$ for which $\mathcal O(x) \cap [\alpha] = \emptyset$.
 {For any finite path  $\alpha$ starting at the root}, let
 	\be 
 	A_\alpha = \{x\in X: \mathcal O(x) \cap [\alpha] = \emptyset\}.
 	\en
 	Then each $A_\alpha$ is invariant (saturated) for the tail equivalence relation. Moreover, we will show that it is closed and nowhere dense, and thus the set $\cup_\alpha A_\alpha$ of points without a dense orbit is meager.}
 	
 	{Suppose that $y \in A_\alpha^c$ and find $z \in \mathcal O(y)$ such that $z \in [\alpha]$. 
 	By connectedness of the diagram, there is a vertex $w$ at some level $m > n$ such that $v_n(y)$ and $v_n(z)$ are connected to $w$. 
 	Now any $y' \in X$ close enough to $y$ that $y'_{[0,m]}=y_{[0,m]}$ also satisfies $\mathcal O(y') \cap [\alpha] \neq \emptyset$, so $y' \in A_\alpha^c$. Therefore $A_\alpha^c$ is open.}
 	
 	{Now we will show that $A_\alpha$ is nowhere dense by showing that its complement $A_\alpha^c$ is dense. 
 	Let $\beta$ be any finite path starting at the root. 
 	Extending either $\alpha$ or $\beta$ if necessary, we may assume that they end at vertices $v_1$ and $v_2$ at the same level $n > 0$. 
 	By connectedness of the diagram, $\alpha$ and $\beta$ have extensions $\alpha'$ and $\beta'$ that end at the same vertex $w$ at some level $m >n$.
 	Extend $\beta'$ to an infinite path $y'$.
 	Then $\mathcal O(y') \cap [\alpha] \neq \emptyset$, and $\mathcal O(y') \cap [\beta] \neq \emptyset$, 
 	so that $y' \in A_\alpha^c \cap [\beta]$.}
 \end{proof}

	\begin{proposition}\label{prop:minmax}
		Suppose that a diagram is connected and every vertex has at least two outgoing edges. 
		Then for every edge ordering the set of orbits of minimal paths is meager, and  
		similarly the set of orbits of maximal paths is meager; 
		therefore, 
		{for every ordering $\xi$ the set} 
		$X(\xi)$ as defined above {(Definition \ref{def:genericpaths})} is comeager.
	\end{proposition}
	\begin{proof}
			Fix an arbitrary edge ordering. 
		For each $k \geq 1$ let 
		\begin{equation} 
		U_k = \{x \in X: \text{ the edge entering } v_k(x) \text{ is not minimal}\}.
		\end{equation}
		Then the set of orbits of minimal paths (the set of paths that eventually follow only minimal edges) is 
		\begin{equation}
	{	E = \bigcup_n \bigcap_{k \geq n} U_k^c.}
		\end{equation}
	   See Figure \ref{fig:E}.

  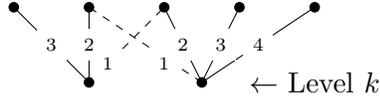
\begin{figure}
      \centering
      \begin{tikzpicture}
      \foreach \n in {0,1}{
      \draw (1,0) -- (\n,1);
      \fill(\n,1) circle (2pt);
      }
      \draw[dashed](1,0)--(2,1);
      \foreach \n in {2,3,4}{
      \draw(2.5,0)--(\n,1);
          \fill(\n,1) circle (2pt);}
          \draw[dashed](2.5,0)--(1,1);
          \fill (1,0) circle (2pt);
          \fill (2.5,0) circle(2pt);
          \node at (4,0){$\leftarrow$ Level $k$};
          \node[fill=white] at (.5,.5){\tiny{$3$}};
          \node[fill=white] at (1,.5){\tiny{$2$}};
          \node[fill=white] at (1.25,.25){\tiny{$1$}};
          \node[fill=white] at (2,.25){\tiny{$1$}};
          \node[fill=white] at (2.25,.5){\tiny{$2$}};
          \node[fill=white] at (2.75,.5){\tiny{$3$}};
          \node[fill=white] at (3.25,.5){\tiny{$4$}};
           \end{tikzpicture}
      \caption{The edge ordering is given by the numbers on the edges. Any infinite paths which contain a dashed edge are in $U_k^c$ and the rest 
      are in $U_k$.}
      \label{fig:E}
  \end{figure}
		To show that $E$ is meager, then, it suffices to show that each of the closed sets 
		$\cap_{k \geq n} U_k^c$ is nowhere dense, equivalently that each of the open sets 
		\begin{equation}
		\mathcal{E}_n = \bigcup_{k \geq n} U_k
		\end{equation}
		is dense.
		
		Given $n \geq 0$ and a basic open set $B$ defined by an initial segment $y$ from the root to  a vertex $v_1$ at a level $m_1$, we find a point $x \in \mathcal{E}_n \cap B$ as follows. 
		Since the sets $\mathcal{E}_n$ are decreasing, we may assume that $n > m_1$.
		Extend $y$ arbitrarily to arrive at a vertex $v_2$ at level $n$.
		The vertex $v_2$ has at least two outgoing edges. 
		Let us suppose first that all edges leaving $v_2$ arrive at the same terminal vertex, $w$. 
		Then one of these edges, call it $e_1$, is not minimal. 
		Extend the finite path downward from $v_2$ along $e_1$ and then arbitrarily downward from $w$, producing a path $x \in \mathcal{E}_{n+1} \cap B\subseteq \mathcal{E}_n\cap B$. 
		
		If there are two edges leaving $v_2$ that arrive at different vertices $w_1, w_2$, 
		since the diagram is connected there is a first level $m_2>n+1$ on which there exists a vertex $w$ that connects to both $w_1$ and $w_2$, along paths $p_1,p_2$ down from $v_2$. 
		One of these two paths, call it $p_1$, enters $w$ along a nonminimal edge. 
		In this case we extend the path downward from $v_2$ along $p_1$ and then arbitrarily downward from $w$ to produce the path $x \in \mathcal{E}_{m_2}\cap B\subseteq \mathcal{E}_n \cap B$. See Figure \ref{fig:EnB}.
		
		This proves the set $X_{\xi}'$ is meager; then, by Prop 2.2, $X(\xi)$ is comeager.
		\end{proof}

  \begin{figure}
      \centering
      \begin{tikzpicture}
          \fill (2,9) circle (3pt);
          \draw(2,9)--++(-.75,-.75)--++(.25,-.25)--++(-.75,-.75)--++(.25,-.25)--++(-1,-1)--++(1,-1)--++(-.25,-.25)--++(.75,-.75)--+(-1,-1);
          \draw(1.5,4)--++(1,-1);
          \node[fill=white] at (1,7.5){$B$};
          \fill(0,6) circle (3pt);
          \node at (1.5,6){$\leftarrow$ Level $m_1$};
          \fill(1.5,4) circle (3pt);
          \node at (-.2,5.75){$v_1$};
          \node at (1.5,3.65){$v_2$};
          \node at (3,4){$\leftarrow$ Level $n$};
          \fill(.5,3) circle (3pt);
          \fill(2.5,3) circle (3pt);
          \node at(0,3){$w_1$};
          \node at (3,3){$w_2$};
          \draw(.5,3)--++(.25,-.25)--++(-.5,-.5)--++(.25,-.25)--++(-.25,-.25)--++(.5,-.5)--++(-.25,-.25);
          \draw(2.5,3)--++(.25,-.25)--++(-.5,-.5)--++(.25,-.25)--++(.25,-.25)--++(-.5,-.5)--++(.25,-.25)--++(-1,-1);
          \draw[dashed](.5,1)--+(1,-1);
          \fill(.5,1) circle (3pt);
          \fill(2.5,1) circle (3pt);
          \fill(1.5,0) circle (3pt);
          \node at (3,0){$\leftarrow$ Level $m_2$};
          \node[fill=white] at(0.5,1.5){\tiny{$p_1$}};
          \node[fill=white] at(2.5,2){\tiny{$p_2$}};
          \node at (1.5,-.25){$w$};
      \end{tikzpicture}
      \caption{The dashed edge into $w$ is non minimal. The path $x\in \mathcal{E}_n\cap B$ follows the path $p_1$ into $w$ through the non minimal edge.}
      \label{fig:EnB} 
  \end{figure}
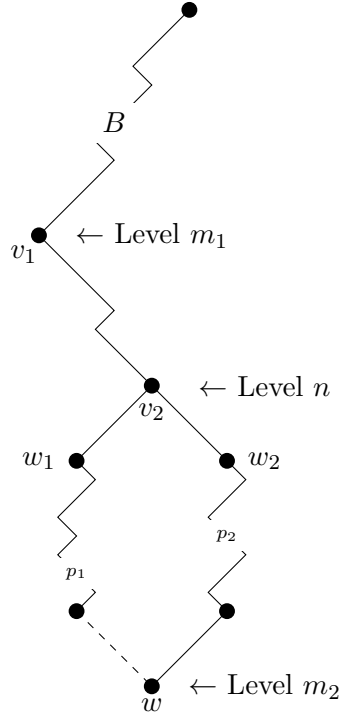

  {
  \begin{remark}\label{rem:measures}
  	The preceding argument shows that for every $k$,  the set of minimal paths $\cap_{k \geq 1} U_k^c$ is nowhere dense (but it can be uncountable---see \cite[Example 7.2]{FPS2017}), similarly for the maximal paths. 
  	Using the argument in \cite[Lemma 6.2]{Berthe2017}, we can show 
  	{(in Proposition \ref{prop:measures})} that 
  	 in any polynomial shape diagram (see Definition \ref{def:polyshape}) the sets of maximal and minimal paths (along with their orbits) are also negligible in the measure-theoretic sense
  	 \and therefore $X_0 \cup X_\xi'$ is meager and has measure $0$ for every fully supported ergodic invariant measure on $X$.
  	Alternatively, in many cases the ergodic measures can be identified explicitly (see M\'{e}la 
  	\cites{xman, Mela2006} and the extension by Frick \cites{Bailey2006, 
  		Frick 2009}, showing that the measure of every cylinder set is given by the product of weights that come from a finite set in $(0,1)$), and then one can apply an argument such as in 
  	\cite[Proposition 2.1]{FPS2017}. 
  	We defer the details to Proposition \ref{prop:measures} in Section \ref{sec:systems}, after the necessary definitions, notations, and properties have been presented.
  \end{remark}
}

We turn now to establish terminology and notation for the main question addressed in this paper, the effectiveness of coding of orbits by finite paths, edges, or vertices.
For each $k \geq 1$ denote by $A_k$ the finite alphabet whose elements are the finite paths (segments, strings of edges) from the root to level $k$. 
	For each $a \in A_k$, the set $[a]=\{x \in X: x_0 \dots x_{k-1}=a\}$ is a clopen cylinder set, and $\mathcal P_k=\{ [a]:a \in A_k\}$ is a partition of $X$ into clopen sets.  
	The map $\pi _k:X \to A_k^\Z$ is defined by $(\pi_kx)_n=a$ if and only if $T^nx \in [a]$.
	{The sequence $\pi_kx$ is called the {\em $k$-coding of $x$.}}
	We denote by $\Sigma_k$ the closure of $\pi_k X$. Denote by $\sigma$ the shift transformation on $A_k^\Z$.   
	Then $\pi_k:(X,T) \to (\Sigma_k, \sigma)$ is a Borel measurable map that commutes with the transformations and $(\Sigma_k, \sigma)$ is a symbolic dynamical system. 

\begin{definition}\label{def:codingbyvertices}
{The {\em coding by vertices at level $j<n$ of a vertex $w$ at level $n$}, denoted by $C_j(w)$, is defined as follows. 
	List in the order determined by their lexicographical ordering in the diagram the paths entering $w$ from vertices at level $j$ as $\{p_1,\dots ,p_r\}$ and denote the source of (the first edge of) $p_i$ by $u_i, i=1,\dots ,r$.
	Then $C_j(w)=u_1 \dots u_r$. 
	(For $j=n-1$, this is the ``morphism read on $\mathcal V_{n}$" in \cite{Durand2010}*{p. 328}).\\
	The {\em coding by edges from level $n-1$ to level $n$ of the orbit of a path $x$} is the sequence $Edge_n(x)=((T^mx)_{n-1}, m\in \mathbb Z)$.\\
  The {\em coding by vertices at level $n$ (or by $\mathcal V_n$) of the orbit of a path} $x$ is the sequence $Vert_n(x)=(v_n(T^mx), m \in \mathbb Z)=((Vert_n(x))_m, m \in \Z)$. }
  \end{definition}
  \begin{remark}\label{rem:codingbyvertices}
{(1) The mapping that takes a path $x$ to its coding by vertices at level $n$ intertwines the action of $T$ on $X$ with the shift on $\mathcal V_n^\Z$.\\
	\noindent
	(2) It might happen that there are $w_1,w_2 \in \mathcal V_n$ with $w_1 \neq w_2$ but, for example, $C_{n-1}(w_1)=C_{n-1}(w_2)$.\\
	\noindent
	(3) Note that whenever $m$ is such that $(T^m x)_{n-1}$ is the minimal edge entering $w=v_n(T^mx)$ from level $n-1$, then the coding of the orbit of $x$ by $\mathcal V_n$ begins a string of repeats of $w$ of length equal to the number of paths from the root to $w$, which we denote by $\dim w$. 
	This is because the coding by vertices $C_{n-1}(v_n(T^mx))=u_1 u_2 \dots$ of $v_n(x)$ by vertices at level $n-1$ expands into the codings of the $u_i$ by vertices at level $n-2$, and each of those vertices has its coding by vertices of the previous level expand to codings by vertices of its previous level, and so on. 
	Stated slightly differently, if we list in their lexicographical order in the diagram the paths from the root entering $w$ as $\{q_1,\dots ,q_{\dim w}\}$, whenever $(T^mx)_{n-1}$ is the minimal edge to $w$ from level $n-1$, 
	as we apply $T$ the orbit of $x$ uses each $q_i$ in turn, and so 
	in the coding by $\mathcal V_n$ of the orbit of $x$ we see a string of length $\dim w$ of repeated $w$'s with the ``dot" to its immediate left (the path being coded passes through the vertex just to the right of the dot).
	}
\end{remark}

In \cite{AFP} a (measure-theoretic) rank one system was defined to be {\em essentially $k$-expansive} if the partition $\mathcal P_k$ generates the full sigma-algebra under the transformation $T$. 
Since we are dealing here with more generality than rank one, in particular not assuming a fixed invariant measure, we need a slightly different definition. 
Also, we restrict attention to points that have dense full two-sided orbits. 

\begin{definition}\label{def:EE}
	We say that a topologically transitive $BV$ system (with ordering $\xi$) is {\em (bilaterally) expansive} if there is an $i \geq 1$ such that the map $\pi_i: (X(\xi),T) \to (\Sigma_i,\sigma)$ is injective, in which case
	the system is called {\em (bilaterally) $i$-expansive}.
\end{definition}
\begin{definition}\label{def:inex}
	We say that a topologically transitive diagram is {\em inherently expansive} if for every ordering the resulting system is (bilaterally) expansive.
	\end{definition} 

\begin{remark}\label{rem:recog}
	The concept of expansiveness for Bratteli-Vershik systems is tantamount to recognizability for substitutions, morphisms, or sequences of morphisms: see \cites{BezuglyiKwiatkowskiMedynets2009, Berthe2017, Beal2023} and their references. 
		These papers focus on systems of bounded width (``finite alphabet rank", possibly after telescoping), and recognizability for aperiodic points, while most of the examples we are interested in have unbounded width.
	The sequence of morphisms associated with a Bratteli-Vershik system is {{\em recognizable}} (at all levels, for full orbits) {in the sense of \cite{Berthe2017}*[Definitions 2.1 and 4.1] if (briefly)}
	for each $n \geq 1$ the coding of an orbit by vertices at level $n$ uniquely determines (by ``desubstitution") its coding by vertices at level $n+1$.
	 We claim that a system is recognizable at all levels for full orbits if and only if it is $1$-expansive for full orbits.
  
	Let us note first that for an ordered Bratteli diagram with at 
 least two vertices per level after the root, for each $i \geq 0$ 
 codings of a path by edges between levels $i$ and $i+1$, and 
 codings by vertices at levels $i$ and $i+1$, determine one 
 another.
 
 In each case we assume that a coding also specifies the location of the central coordinate (``dot") in the given sequence of edges or symbols.
	A sequence of edges $(T^mx)_i, m \in \Z$, clearly determines the sequences of vertices $v_i(T^mx)=u_m(x)$ and $v_{i+1}(T^mx)=w_m(x)$.
	Conversely, given the sequences {(on $m$)} 
 $v_i(T^mx)=u_m(x)$ and {$v_{i+1}(T^mx)=w_m(x)$} and the position of the dot in each sequence, {in a string of repeats of a symbol $w_m(x)$ the path is entering that vertex along the edge whose label is given by the position of the dot. Thus} the codings by vertices determine for each $m \in \Z$ a unique edge $e \in \mathcal E_i$ such that $s(e)=u_m(x)$ and $t(e)=w_m(x)$. 

 {Suppose that} $X \setminus X_\xi$ is $1$-expansive and we are given the coding $(v_n(T^mx), m \in \Z)$, of a path $x$ by vertices at some level $n \geq 1$. 
	By Remark \ref{rem:codingbyvertices} (3), each coding of $x$ by $\mathcal V_{n}$ determines its coding by $\mathcal V_{n-1}$, using the ``substitution read on $\mathcal V_{n}$", including the position of the ``dot''.
	Thus the coding by $\mathcal V_{n}$ determines the codings by vertices on all levels $j \leq n$, and hence, by the paragraph above, the coding by edges for all levels before level $n$, {in particular the $1$-coding,} hence (by $1$-expansiveness) the path $x$ itself, and hence the coding by vertices at level $n+1$.

     Conversely, if $X \setminus X'_\xi$ is recognizable at all levels, then for each path $x \in X \setminus X'_\xi$ 
		the coding by vertices at each level $n \geq 1$ of the 
  orbit of $x$ determines its coding by vertices at level $n+1$, 
  {so we may just 
  reverse the process in the previous paragraph:
  If $X \setminus X'_\xi$ is recognizable at all levels, then coding by edges to level $1$, which determines coding by vertices at level $1$, by recognizability determines the coding by vertices at all levels, hence the coding by edges at all levels, hence the edge traversed by each path at every level, hence each path itself, so that the system is $1$-expansive.}

	\end{remark} 
If a system is not $i$-expansive, there exist pairs of points in $X \setminus X'_\xi$ such that at all times in their orbits they agree on their first $i$ edges. The following definition is from \cite{dm2008}.

\begin{definition}\label{def:depth}
	 Let $i\geq 0$. We say two paths $x, x' \in X \setminus X'_\xi$ form a pair of {\em depth $i$} if they have the same $i$-coding but not the same $(i+1)$-coding. \end{definition}

	\begin{proposition}\label{prop:distinct}
Consider an ordered Bratteli diagram for which there are at least two vertices at every level after the root.
		 Then for each $i \geq 1$ and every depth $i$ pair, $x$ and $x'$ in $X(\xi)$, there is $m \in \Z$ such that 
 $v_{i+1}(T^mx) \neq v_{i+1}(T^mx')$.
\end{proposition}
\begin{proof}
	Since the pair $x,x'$ is depth $i$, there is $j \in \Z$ such that $(T^jx)_i \neq (T^jx')_i$. 
		If it happens that $v_{i+1}(T^jx)=w=v_{i+1}(T^jx')$, one of $x,x'$ must reach 
  the maximal path to $w$ and change vertices at level $i+1$ before the other, since 
  there are at least two vertices at every level and both $x$ and $x'$ have dense 
  orbits. Thus there is a smallest $k \geq 0$ such that $v_{i+1}(T^{j+k}x) \neq 
  v_{i+1}(T^{j+k}x')$. Letting $m=k+j$, we have $v_{i+1}(T^mx) \neq v_{i+1}
  (T^mx')$.
\end{proof}

\section{Multivariable polynomial Bratteli diagrams}\label{sec:systems}
NOTE: In this section we use the notation $x_i$ for variables (the arguments in multivariable polynomials) and {\em not} for the edges of paths in the diagram.

 A homogeneous positive integer multivariable polynomial of degree $d$ in $q$ variables
 \[p(x)=p(x_1,\dots, x_q) =\sum_{m_1+m_2+\dots m_q=d}a_{(m_1,m_2,\dots,m_q)}x_1^{m_1}x_2^{m_2}\dots x_q^{m_q}\] 
 where $m_1,\dots m_q$ are non-negative integers and $a_{(m_1,m_2,\dots, m_q)}\in \N$,    
defines {a Bratteli diagram} in the following manner. 
	Each 
vertex {$w \in \mathcal V_n$} is a momomial $x_1^{m_1}\dots x_q^{m_q}$ in $(p(x))^n$, with $\sum_{i=1}^qm_i=nd$, {which we also identify} with the vector
 $(m_1,\dots, m_q)=\sum_{i=1}^q 
m_ie_i$, where $e_i$ is the $i$'th standard basis vector in 
$q$ space.  We write $w=(m_1,\dots,m_q)=(w(1),\dots,w(q))$. By $w\geq 0$ we mean $w(\ell)\geq 0$ for all  $\ell \in\{1, ...q\}$, and $w\geq w'$, $w-w'\geq 0$, {and $|w|=w_1+w_2+\dots +w_q$.} 

For each $n\in \N$, the 
number of vertices in $\mathcal{V}_n$ equals the number of terms in 
$(p(x))^n$, equivalently, the number of ways $nd$ can be 
written as a sum of $q$ nonnegative integers. In other words,
\[|\mathcal{V}_n|=\binom{nd+q-1}{q-1}.\]

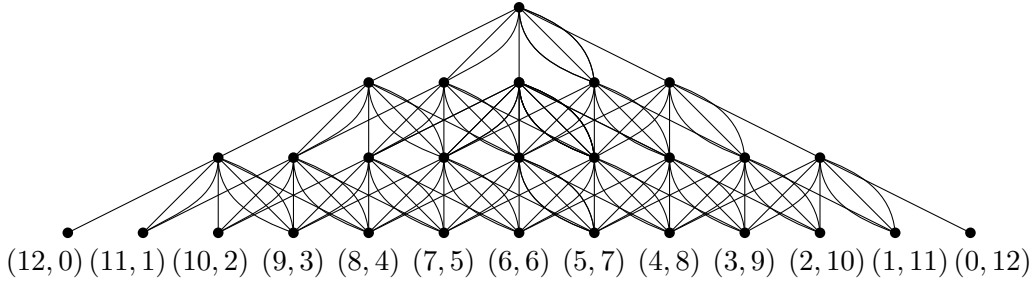
\begin{figure}
\begin{tikzpicture}[scale=1]
\foreach \n in {1,2}{
	\fill(0,-\n) circle (2pt);
	\fill(0,-\n+1) circle (2pt);
	\draw(0,-\n)--(-2,-\n-1);
	\draw(0,-\n)--(-1,-\n-1);
	\draw(0,-\n)--(0,-\n-1);
	\draw(0,-\n)--(1,-\n-1);
	\draw(0,-\n)--(2,-\n-1);
	\draw  (0,-\n) to [out=270,in=30] (0-1,-\n-1);
	\draw(0,-\n+1)--(-2,-\n);
	\draw(0,-\n+1)--(-1,-\n);
	\draw(0,-\n+1)--(0,-\n);
	\draw(0,-\n+1)--(1,-\n);
	\draw(0,-\n+1)--(2,-\n);
	\draw  (0,-\n+1) to [out=270,in=30] (-1,-\n);
	\draw (0,-\n) to [out=-20,in=95] (1,-\n-1);
	\draw (0,-\n) to [out=-20,in=95] (1,-\n-1);
	\draw  (0,-\n) to [out=270,in=160] (1,-\n-1);
	\draw  (0,-\n) to [out=270,in=160] (1,-\n-1);
	\draw (0,-\n+1) to [out=-20,in=95] (1,-\n);
	\draw (0,-\n+1) to [out=-20,in=95] (1,-\n);
	\draw  (0,-\n+1) to [out=270,in=160] (1,-\n);
	\draw  (0,-\n+1) to [out=270,in=160] (1,-\n);
	\pgfmathsetmacro{\p}{2*\n}
	\foreach \k in {1,2,...,\p}{
		\fill(\k,-\n) circle (2pt);
		\fill(-\k,-\n) circle (2pt);
		\draw(\k,-\n)--(\k-2,-\n-1);
		\draw(-\k,-\n)--(-\k-2,-\n-1);
		\draw(-\k,-\n) to [out=270,in=30] (-\k-1,-\n-1);
		\draw (\k,-\n) to [out=270,in=30] (\k-1,-\n-1);
		\draw (-\k,-\n) to [out=-20,in=95] (-\k+1,-\n-1);
		\draw (\k,-\n) to [out=-20,in=95] (\k+1,-\n-1);
		\draw  (-\k,-\n) to [out=270,in=160] (-\k+1,-\n-1);
		\draw  (\k,-\n) to [out=270,in=160] (\k+1,-\n-1);
		\draw(\k,-\n)--(\k-1,-\n-1);
		\draw(-\k,-\n)--(-\k-1,-\n-1);
		\draw(\k,-\n)--(\k,-\n-1);
		\draw(-\k,-\n)--(-\k,-\n-1);
		\draw(\k,-\n)--(\k+1,-\n-1);
		\draw(-\k,-\n)--(-\k+1,-\n-1);
		\draw(\k,-\n)--(\k+2,-\n-1);
		\draw(-\k,-\n)--(-\k+2,-\n-1);
  }}
    \foreach \k in {-6,-5,...,6}{
    \fill(\k,-3) circle (2pt);}
    \node at (-6.3,-3.4){$(12,0)$};
    \node at (-5.2,-3.4){$(11,1)$};
    \node at (-4.1,-3.4){$(10,2)$};
    \node at (-3,-3.4){$(9,3)$};
    \node at (-2,-3.4){$(8,4)$};
    \node at (-1,-3.4){$(7,5)$};
    \node at (0,-3.4){$(6,6)$};
    \node at (1,-3.4){$(5,7)$};
    \node at (2,-3.4){$(4,8)$};
    \node at (3,-3.4){$(3,9)$};
    \node at (4.1,-3.4){$(2,10)$};
    \node at (5.2,-3.4){$(1,11)$};
    \node at (6.3,-3.4){$(0,12)$};
        
\end{tikzpicture}

    \caption{The polynomial adic system associated with $$p(x,y)=x^4+2x^3y+x^2y^2+3xy^3+y^4,$$ has $q=2$ variables and is of degree $d=4$. The vertices on level $3$ corresponding to terms in $(p(x,y))^3$ are labeled.}
    \label{fig:MVadic}
\end{figure}

\begin{definition}
    The set of \emph{source vectors} is $S=\{s=(s_1,s_2,\dots,s_q):s\geq 0 \text{ and } {|s|}=d\}$. 
\end{definition}
Thinking of $\mathcal{V}_1$ as a set of vectors, we have that $\mathcal{V}_1=S$, and hence each $s\in S$ also corresponds to a monomial in $p(x)$. {In this way we can also think about an edge, $s$, connecting vertices $w\in \mathcal{V}_n$ and $v\in \mathcal{V}_{n+1}$ as multiplication of the monomial from $(p(x))^n$ corresponding to $w$ by the monomial corresponding to $s$, resulting in the monomial in $(p(x))^{n+1}$ correspoding to $v$.}
  
 \begin{definition}\label{def:polyshape}
  Let $p(x)=p(x_1,\dots,x_q)$ be a homogeneous positive integer multivariable polynomial. Consider a Bratteli diagram $\mathcal B=(\mathcal{V},\mathcal{E})$ for which the set of vertices is defined by $p$ as above. Suppose that for each $n \geq 0$ the set $\mathcal E_n$ of edges connecting vertices in $\mathcal V_n$ to vertices in $\mathcal V_{n+1}$ has the property that there is at least one edge from $u \in \mathcal V_{n}$ to $w \in \mathcal V_{n+1}$ if and only if there is $s \in S$ such that $u+s=w$. 
We then say that the diagram is (or has) {\em polynomial shape} and {\em is associated to $p$.}
If in addition the number of edges between $u$ and $w$ is the coefficient of the monomial $x_1^{s_1}x_2^{s_2}\dots x_q^{s_q}$ in $p$ that corresponds to the vector $s=w-u$, we say that the diagram is {\em a polynomial diagram} and {\em is defined by $p$}.
\end{definition}

In a polynomial diagram, the number of paths from the root vertex to a vertex $u=(m_1,m_2,\dots,m_q)$ is the coefficient of the monomial $x_1^{m_1}x_2^{m_2}\dots 
	x_q^{m_q}$ in $(p(x_1,\dots,x_q))^{n}$. See Figure \ref{fig:MVadic}. 
	
	\begin{definition}\label{def:sourceset}
	{For $n \geq 1$ and $w \in \mathcal V_{n+1}$ we define the {\em source set} of $w$ to be}
		\be
	{S(w)=\{u \in \mathcal V_{n}: w-u \in S\}.}
		\en
		\end{definition}
		Thus $u \in S(w)$ if and only if there is an edge from $u$ to $w$.
		Since vector 
addition corresponds to multiplication of monomials, $u\in S(w)$ if and only if the monomial corresponding to $u$ can be multiplied by a {monomial} of 
$p(x)$ to obtain the monomial corresponding to $w$.

Every polynomial diagram has polynomial shape, but a polynomial shape diagram associated to a polynomial $p$, while having connections between the same pairs of vertices as in the polynomial diagram defined by $p$, can have any positive number of edges between connected vertices.
When $q=1$, a polynomial shape diagram is that of either an odometer or a finite set of points.

{\em For the rest of this work we will assume that $X$ refers to the path space of a polynomial shape diagram for which $q\geq 2$.}
 This class of diagrams satisfies all the conditions mentioned in Section \ref{sec:background}.
  Some polynomial shape diagrams that have previously been studied include the Euler and reverse Euler diagrams in any dimension as well as the Stirling diagrams. 
See Figure \ref{fig:Euler}.

\begin{figure}
\begin{tikzpicture}[scale=.4]
\draw(-1,-1)--+(-1,-1);
\foreach \k in {-1,0,1,2}{
\draw[rounded corners](\k,-2-\k)--+(.25,-.75)--+(1,-1);
\draw[rounded corners](\k,-2-\k)--+(.75,-.25)--+(1,-1);
\draw[rounded corners](-\k,-2-\k)--+(-.25,-.75)--+(-1,-1);
\draw[rounded corners](-\k,-2-\k)--+(-.75,-.25)--+(-1,-1);}
\foreach \k in {-2,-1,0}{
\draw(\k,-4-\k)--+(1,-1);
\draw[rounded corners](\k,-4-\k)--+(.25,-.75)--+(1,-1);
\draw[rounded corners](\k,-4-\k)--+(.75,-.25)--+(1,-1);
\draw(-\k,-4-\k)--+(-1,-1);
\draw[rounded corners](-\k,-4-\k)--+(-.25,-.75)--+(-1,-1);
\draw[rounded corners](-\k,-4-\k)--+(-.75,-.25)--+(-1,-1);
}
\foreach \k in {-3,-2}{
\draw[rounded corners](\k,-6-\k)--+(.25,-.75)--+(1,-1);
\draw[rounded corners](\k,-6-\k)--+(.75,-.25)--+(1,-1);
\draw[rounded corners](-\k,-6-\k)--+(-.25,-.75)--+(-1,-1);
\draw[rounded corners](-\k,-6-\k)--+(-.75,-.25)--+(-1,-1);
\draw[rounded corners](\k,-6-\k)--+(.38,-.62)--+(1,-1);
\draw[rounded corners](\k,-6-\k)--+(.62,-.38)--+(1,-1);
\draw[rounded corners](-\k,-6-\k)--+(-.38,-.62)--+(-1,-1);
\draw[rounded corners](-\k,-6-\k)--+(-.62,-.38)--+(-1,-1);
}
\foreach \k in {-4}{
\draw(\k,-8-\k)--+(1,-1);
\draw(-\k,-8-\k)--+(-1,-1);
\draw[rounded corners](\k,-8-\k)--+(.25,-.75)--+(1,-1);
\draw[rounded corners](\k,-8-\k)--+(.75,-.25)--+(1,-1);
\draw[rounded corners](-\k,-8-\k)--+(-.25,-.75)--+(-1,-1);
\draw[rounded corners](-\k,-8-\k)--+(-.75,-.25)--+(-1,-1);
\draw[rounded corners](\k,-8-\k)--+(.38,-.62)--+(1,-1);
\draw[rounded corners](\k,-8-\k)--+(.62,-.38)--+(1,-1);
\draw[rounded corners](-\k,-8-\k)--+(-.38,-.62)--+(-1,-1);
\draw[rounded corners](-\k,-8-\k)--+(-.62,-.38)--+(-1,-1);
}
\draw(0,0)--(-5,-5);
\draw(0,0)--(5,-5);
\fill (0,0) circle (3pt);

\foreach \k in {-1,1}{
\fill (\k,-1) circle (3pt);
}
\foreach \k in {-2,0,2}{
\fill (\k,-2) circle (3pt);
}
\foreach \k in {-3,-1,1,3}{
\fill (\k,-3) circle (3pt);
}
\foreach \k in {-4,-2,0,2,4}{
\fill (\k,-4) circle (3pt);
}
\foreach \k in {-5,-3,-1,1,3,5}{
\fill (\k,-5) circle (3pt);
}

\end{tikzpicture}\hskip .2in\begin{tikzpicture}[scale=.4]
\fill(0,0) circle (3pt);
\draw(0,0)--++(-1,-1);
\draw(0,0)--++(1,-1);
\foreach \k in {1,2,3,4}{
\draw(\k,-\k)--+(-1,-1);
\draw(-\k,-\k)--+(1,-1);
}
\foreach \k in {-1,0,1,2}{
\draw[rounded corners](\k,-2-\k)--+(-.25,-.75)--+(-1,-1);
\draw[rounded corners](\k,-2-\k)--+(-.75,-.25)--+(-1,-1);
\draw[rounded corners](-\k,-2-\k)--+(.25,-.75)--+(1,-1);
\draw[rounded corners](-\k,-2-\k)--+(.75,-.25)--+(1,-1);}
\foreach \k in {-2,-1,0}{
\draw(\k,-4-\k)--+(-1,-1);
\draw[rounded corners](\k,-4-\k)--+(-.25,-.75)--+(-1,-1);
\draw[rounded corners](\k,-4-\k)--+(-.75,-.25)--+(-1,-1);
\draw(-\k,-4-\k)--+(1,-1);
\draw[rounded corners](-\k,-4-\k)--+(.25,-.75)--+(1,-1);
\draw[rounded corners](-\k,-4-\k)--+(.75,-.25)--+(1,-1);
}
\foreach \k in {-3,-2}{
\draw[rounded corners](\k,-6-\k)--+(-.25,-.75)--+(-1,-1);
\draw[rounded corners](\k,-6-\k)--+(-.75,-.25)--+(-1,-1);
\draw[rounded corners](-\k,-6-\k)--+(.25,-.75)--+(1,-1);
\draw[rounded corners](-\k,-6-\k)--+(.75,-.25)--+(1,-1);
\draw[rounded corners](\k,-6-\k)--+(-.38,-.62)--+(-1,-1);
\draw[rounded corners](\k,-6-\k)--+(-.62,-.38)--+(-1,-1);
\draw[rounded corners](-\k,-6-\k)--+(.38,-.62)--+(1,-1);
\draw[rounded corners](-\k,-6-\k)--+(.62,-.38)--+(1,-1);
}
\foreach \k in {-4}{
\draw(\k,-8-\k)--+(-1,-1);
\draw(-\k,-8-\k)--+(1,-1);
\draw[rounded corners](\k,-8-\k)--+(-.25,-.75)--+(-1,-1);
\draw[rounded corners](\k,-8-\k)--+(-.75,-.25)--+(-1,-1);
\draw[rounded corners](-\k,-8-\k)--+(.25,-.75)--+(1,-1);
\draw[rounded corners](-\k,-8-\k)--+(.75,-.25)--+(1,-1);
\draw[rounded corners](\k,-8-\k)--+(-.38,-.62)--+(-1,-1);
\draw[rounded corners](\k,-8-\k)--+(-.62,-.38)--+(-1,-1);
\draw[rounded corners](-\k,-8-\k)--+(.38,-.62)--+(1,-1);
\draw[rounded corners](-\k,-8-\k)--+(.62,-.38)--+(1,-1);
}

\foreach \k in {-1,1}{
\fill (\k,-1) circle (3pt);
}
\foreach \k in {-2,0,2}{
\fill (\k,-2) circle (3pt);
}
\foreach \k in {-3,-1,1,3}{
\fill (\k,-3) circle (3pt);
}
\foreach \k in {-4,-2,0,2,4}{
\fill (\k,-4) circle (3pt);
}
\foreach \k in {-5,-3,-1,1,3,5}{
\fill (\k,-5) circle (3pt);
}

\end{tikzpicture}\hskip .2in \begin{tikzpicture}[scale=.4]
\fill(0,0) circle (3pt);
\draw(0,0)--++(-1,-1);
\draw(0,0)--++(1,-1);
\foreach \k in {-1,1}{
\draw[rounded corners](\k,-1)--+(-.25,-.75)--+(-1,-1);
\draw[rounded corners](\k,-1)--+(-.75,-.25)--+(-1,-1);
\draw(\k,-1)--+(1,-1);
}
\foreach \k in {-2,0,2}{
\draw[rounded corners](\k,-2)--+(-.25,-.75)--+(-1,-1);
\draw[rounded corners](\k,-2)--+(-.75,-.25)--+(-1,-1);
\draw(\k,-2)--+(-1,-1);
\draw(\k,-2)--+(1,-1);
}
\foreach \k in {-3,-1,1,3}{
\draw[rounded corners](\k,-3)--+(-.25,-.75)--+(-1,-1);
\draw[rounded corners](\k,-3)--+(-.75,-.25)--+(-1,-1);
\draw[rounded corners](\k,-3)--+(-.38,-.62)--+(-1,-1);
\draw[rounded corners](\k,-3)--+(-.62,-.38)--+(-1,-1);
\draw(\k,-3)--+(1,-1);
}
\foreach \k in {-4,-2,0,2,4}{
\draw[rounded corners](\k,-4)--+(-.25,-.75)--+(-1,-1);
\draw[rounded corners](\k,-4)--+(-.75,-.25)--+(-1,-1);
\draw[rounded corners](\k,-4)--+(-.38,-.62)--+(-1,-1);
\draw[rounded corners](\k,-4)--+(-.62,-.38)--+(-1,-1);
\draw(\k,-4)--+(-1,-1);
\draw(\k,-4)--+(1,-1);
}
\foreach \k in {-1,1}{
\fill (\k,-1) circle (3pt);
}
\foreach \k in {-2,0,2}{
\fill (\k,-2) circle (3pt);
}
\foreach \k in {-3,-1,1,3}{
\fill (\k,-3) circle (3pt);
}
\foreach \k in {-4,-2,0,2,4}{
\fill (\k,-4) circle (3pt);
}
\foreach \k in {-5,-3,-1,1,3,5}{
\fill (\k,-5) circle (3pt);
}
\end{tikzpicture}
    \caption{The Euler, Reverse Euler, and Stirling Bratteli diagrams.}
    \label{fig:Euler}
\end{figure}
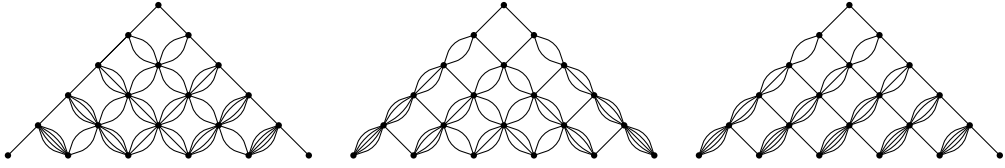

    \begin{remark}\label{rem:MaxCoordinate} 
    If $u\in S(w)$, then for all $j\in \{1,2,\dots, q\}$, 
    $u(j)\leq w(j)\leq u(j)+d$. This implies that if $w,w'\in 
    \mathcal{V}_{n+1}$ and $S(w)\cap S(w')\neq \emptyset$, then for each 
    $j\in \{1,\dots q\}$, $|w(j)-w'(j)|\leq d$. Likewise, if 
    $u,u'\in S(w)$, then for each $j\in \{1,\dots q\}$, $|u(j)-
    u'(j)|\leq d$. Further, the coordinates of each vertex in 
    $\mathcal{V}_{n+1}$ must sum to 
    $(n+1)d$, and since there are exactly $q$ coordinates, at least 
    one coordinate is at least $(n+1)d/q$. 
    \end{remark}  
    The following proposition gives a complete description of which paths in a polynomial shape diagram have dense orbits. The proof is straightforward and is left to the reader.
    \begin{proposition}\label{prop:orbits}
    	In a polynomial shape diagram with path space $X$, a path $x \in X$ has dense orbit if and only if 
    		\be 
    	\lim_{n \to \infty} \mathrm{min} \{v_n(x)(j): j=1, \dots, q\} = \infty. 
    		\en 
    			\end{proposition}
      
{
\begin{proposition}\label{prop:measures}
	For polynomial shape diagrams, for every ordering the sets 
	of maximal and minimal paths have measure zero for every fully 
	supported ergodic invariant (Borel probability) measure.
\end{proposition}}
\begin{proof}
	{As mentioned above in Remark \ref{rem:measures}, we can adapt an argument from \cite{Berthe2017}.} {In \cite{Berthe2017} a diagram was defined to be {\em everywhere growing} if $\min\{\dim v: v \in \mathcal V_n\} \to \infty$ as $n \to \infty$.
	Polynomial shape diagrams need not have this property, but {we will show that} they {almost do, 
		in the sense that}
	for each such diagram 
	there is a finite (possibly empty) set of paths $x^{(i)}, i=1, \dots,q,$ for which $\dim v_n(x^{(i)})$ is bounded {from above}, {while $\min\{\dim v: v \in \mathcal V_n \setminus \cup_{1 \leq i \leq q} v_n(x^{(i)})\} \to \infty$ as $n \to \infty$.}}
	
	Let $\xi$ be an ordering of a polynomial shape diagram $\mathcal B$ that is associated to a polynomial $p$, 
		and let $\mu$ be a fully supported ergodic measure on {the path space} $X$. Since larger coefficients lead to larger dimensions for the vertices,
		 we may assume that all coefficients of $p$ are equal to $1$, so that 
	\be 
	p(x_1, \dots, x_q) = \sum_{n_1+\dots+n_q=d} x_1^{n_1} \cdots x_q^{n_q}.
	\en 
	For each $i=1, \dots, q$ consider the path $x^{(i)}$ such that $v_n(x^{(i)}) = x_i^{nd} \sim nd e_i$ for all $n \geq 1$.
	Then $\dim v_n(x^{(i)})=1$ for each $i$ and $n$. 
	Since $\mu$ is a fully supported ergodic measure on $X$, $\mu\{x^{(i)}\}=0$ for each $i=1, \dots, q$.
	So to show that for any ordering the sets of minimal and maximal paths have measure $0$, we will focus on paths not in this finite set.
	
	As just mentioned above, the dimension of vertex $v=(m_1, \dots, m_q)$ (with $\sum_i m_i=nd)$ at level $n$ is the coefficient of $x_1^{m_1} \cdots x_q^{m_q}$ in $p(x_1, \dots, x_q)^n$. 
	If $v=(m_1, \dots, m_q)$ is not one of the vertices $x_i^{nd}$ (that is, at least two of the $m_i$ are not equal to $0$), then $\dim v \geq n$.

 (This may be seen as follows. The statement is obviously true for $n=1$. 
Any vertex $v=x_1^{m_1} \cdots x_q^{m_q}$ at level $n>1$ is not one of the special $x_i^{nd}$ if and only if it has at least two of the exponents $m_i$ positive; equivalently, if and only if it has at least two different source vertices. 
Thus the statement is also true for $n=2$.
				Assume now that $n>2$, that $\dim w \geq n-1$ for all $w \in \mathcal V_{n-1}$ {for which $w$ is not one of the special $x_i^{nd}$}, and that $v \in \mathcal V_n$
	has at least two source vertices, $v_{n-1}$ and $v'_{n-1}$.
	Then at least one of these, call it $v_{n-1}$, again has two source vertices, since otherwise both would equal some $x_i^{(n-1)d}$, {and this cannot occur for $n>2$.} 
Since $\dim v \geq \dim v_{n-1} + \dim v'_{n-1}$, and $\dim v_{n-1} \geq n-1$, we have $\dim v \geq n$.)

	Let $v \in \mathcal V_{n+1}$ be a vertex at level $n +1$.
Each path $\alpha = \alpha_0 \dots \alpha_n$ from the root to $v$ defines {a} cylinder set $[\alpha]=\{x \in X: x_{[0,n]}=\alpha\}$, and all these cylinder sets have equal measure
	\be 
	\mu[\alpha] = \frac{\mu [v]}{\dim v},
	\en 
	where $[v]=\{x \in X: v_{n+1}(x)=v\}$.
	Denote by 
	{$\mathcal V_{n+1}'$ the set of vertices in $\mathcal V_{n+1}$
		 that are not any of the $(n+1)de_i$ and by } 
	$M_n$ the set of finite minimal paths with terminal vertex in $\mathcal V_{n+1}'$.
	If $x$ is a minimal path that is not one of the $x^{(i)}$, then $x_{[0,n]} \in M_n$ for all large enough $n$. 
    {For each $v \in \mathcal V_{n+1}$} there is exactly one finite minimal path from the root to $v$, so, as in \cite[Lemma 6.2]{Berthe2017},
	\be 
	\mu \bigcup \{ [\alpha]: \alpha \in M_n \} \leq \sum_{v \in \mathcal V_{n+1}'} \frac{\mu [v]}{\dim v}
	\leq \frac{1}{n+1}\sum_{v \in \mathcal V_{n+1}'} \mu[v] \leq \frac{1}{n+1} \to 0 \text{ as } n \to \infty.
	\en 
	Similarly, the set of maximal paths has measure $0$.
  \end{proof}

\begin{definition}\label{def:dsv}
    For each $j\in \{1,\dots,q\}$ let $e_j$ denote the standard $j'th$ basis vector in $q$-space. Given 
    $w\in \mathcal{V}_{n+1}$ the 
    \emph{distinguished source vertex of $w$ in direction $j$}, 
    denoted $DSV(w,j)$ is the vertex $w-de_j\in \mathcal{V}_n$. 
\end{definition}

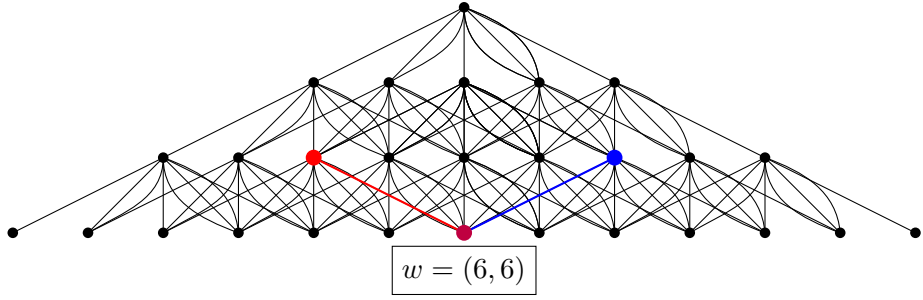
\begin{figure}
\begin{tikzpicture}[scale=1]
\foreach \n in {1,2}{
	\fill(0,-\n) circle (2pt);
	\fill(0,-\n+1) circle (2pt);
	\draw(0,-\n)--(-2,-\n-1);
	\draw(0,-\n)--(-1,-\n-1);
	\draw(0,-\n)--(0,-\n-1);
	\draw(0,-\n)--(1,-\n-1);
	\draw(0,-\n)--(2,-\n-1);
	\draw  (0,-\n) to [out=270,in=30] (0-1,-\n-1);
	\draw(0,-\n+1)--(-2,-\n);
	\draw(0,-\n+1)--(-1,-\n);
	\draw(0,-\n+1)--(0,-\n);
	\draw(0,-\n+1)--(1,-\n);
	\draw(0,-\n+1)--(2,-\n);
	\draw  (0,-\n+1) to [out=270,in=30] (-1,-\n);
	\draw (0,-\n) to [out=-20,in=95] (1,-\n-1);
	\draw (0,-\n) to [out=-20,in=95] (1,-\n-1);
	\draw  (0,-\n) to [out=270,in=160] (1,-\n-1);
	\draw  (0,-\n) to [out=270,in=160] (1,-\n-1);
	\draw (0,-\n+1) to [out=-20,in=95] (1,-\n);
	\draw (0,-\n+1) to [out=-20,in=95] (1,-\n);
	\draw  (0,-\n+1) to [out=270,in=160] (1,-\n);
	\draw  (0,-\n+1) to [out=270,in=160] (1,-\n);
	\pgfmathsetmacro{\p}{2*\n}
	\foreach \k in {1,2,...,\p}{
		\fill(\k,-\n) circle (2pt);
		\fill(-\k,-\n) circle (2pt);
		\draw(\k,-\n)--(\k-2,-\n-1);
		\draw(-\k,-\n)--(-\k-2,-\n-1);
		\draw(-\k,-\n) to [out=270,in=30] (-\k-1,-\n-1);
		\draw (\k,-\n) to [out=270,in=30] (\k-1,-\n-1);
		\draw (-\k,-\n) to [out=-20,in=95] (-\k+1,-\n-1);
		\draw (\k,-\n) to [out=-20,in=95] (\k+1,-\n-1);
		\draw  (-\k,-\n) to [out=270,in=160] (-\k+1,-\n-1);
		\draw  (\k,-\n) to [out=270,in=160] (\k+1,-\n-1);
		\draw(\k,-\n)--(\k-1,-\n-1);
		\draw(-\k,-\n)--(-\k-1,-\n-1);
		\draw(\k,-\n)--(\k,-\n-1);
		\draw(-\k,-\n)--(-\k,-\n-1);
		\draw(\k,-\n)--(\k+1,-\n-1);
		\draw(-\k,-\n)--(-\k+1,-\n-1);
		\draw(\k,-\n)--(\k+2,-\n-1);
		\draw(-\k,-\n)--(-\k+2,-\n-1);
  }}
  \foreach \k in {-6,-5,...,6}{\fill(\k,-3) circle (2pt);}
  \draw[red,thick](-2,-2)--(0,-3);
    \draw[blue,thick](2,-2)--(0,-3);
  \fill[purple](0,-3) circle (3pt);
  \fill[red](-2,-2) circle (3pt);
  \fill[blue](2,-2) circle (3pt);
  \node[fill=white] at (0,-3.5)[rectangle,draw]{$w=(6,6)$};
\end{tikzpicture}
    \caption{$DSV(w,1)=(2,6)$ is shown in blue, and $DSV(w,2)=
    (6,2)$ is shown in red. Another vertex with $DSV(w,1)$ in 
    its source set has first coordinate less than $w(1)=6$. 
    Likewise, a vertex with $DSV(w,2)$ in its source set has 
    second coordinate less than $w(2)=6$.}
    \label{fig:DSVCoordinate}
\end{figure}

\begin{remark}\label{rem:CondForDSV}
A vertex $w\in \mathcal{V}_{n+1}$ has a distinguished source vertex in 
direction $j$ if and only if $w(j)\geq d$. The $DSV(w,j)$ is 
the vertex in the source set of $w$ 
obtained by division of the monomial $x_1^{w(1)}\dots x_q^{w(q)}$ by 
$x_j^d$, reducing the exponent of $x_j$ by the maximal 
possible amount while the other exponents are unchanged. See 
Figure \ref{fig:DSVCoordinate}.\end{remark}

\begin{lemma}\label{lem:BoundedMovement}
    Let $w_0,w_1\in \mathcal{V}_n$ with $w_0\neq w_1$. If $DSV(w_0,j)\in S(w_1)$, then  $$1\leq w_0(j)-w_1(j).$$ 
\end{lemma}
\begin{proof}
    Assume $DSV(w_0,j)\in S(w_1)$. Then there exists ${s}\in S$ such that $DSV(w_0,j)+{s}={w_1}$. This gives
    \begin{align*}
        {w_0}-d{e_j}+{s}={w_1},\\
        \text{so that } {s}={w_1}-{w_0}+d{e_j}.
    \end{align*}

    \begin{figure}
        \centering
        \begin{tikzpicture}
            \fill (0,0) circle (2pt);
            \node[below] at(0,0){$w_0$};
            \draw(0,0)--(2,2)--(2,0);
            \fill(2,0) circle (2pt);
            \node[below]at (2,0){$w_1$};
            \node[fill=white] at (2,2)[rectangle,draw]{$DSV(w_0,j)$};
            \node[fill=white] at (1,1){$de_j$};
            \node[fill=white] at (2,1){$s$};
        \end{tikzpicture}
        \caption{The $DSV(w_0,j)$ is in the source set of $w_1$ and hence $s(j)<d$, which implies $w_1(j)<w_0(j)$.}
        \label{fig:BoundedMovement}
    \end{figure}
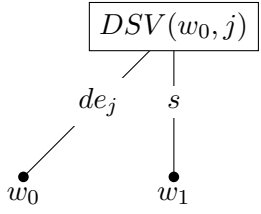
    Since the coordinates of $s$ are nonnegative and sum to $d$, $0\leq {s}(j)= {w_1}(j)-{w_0}(j)+d\leq d$, implying that 
    \[0\leq {w_0}(j)-{w_1}(j).\]
If ${w_0}(j)-{w_1}(j)=0$, this implies that 
${s}=d{e_j}$. 
However, then ${w_1}=DSV(w_0,j)+d{e_j}={w_0}.$ Since we 
assumed that ${w_0}\neq {w_1}$, we must have $w_0(j)-w_1(j) \neq 0$.
 Hence, $$1\leq w_0(j)-
w_1(j)$$ See Figure \ref{fig:BoundedMovement}.
\end{proof}

\begin{lemma}\label{lem:2DSV}
    Let $w_0,w_1\in \mathcal{V}_n$. If $DSV(w_0,j_1)\in S(w_1)$ and $DSV(w_0,j_2)\in S(w_1)$, where $j_1\neq j_2$, then $w_1=w_0$. 
\end{lemma}
\begin{proof}
     Recall from the definition of $DSV(w,j)$ that  \[    w(i)=\begin{cases}
         DSV(w,j)(i) & i\neq j\\
         d+DSV(w,j)(i) & i=j. 
     \end{cases} \]
     Therefore, $w_0(j_1)=DSV(w_0,j_2)(j_1)$. Since $DSV(w_0,j_2)\in S(w_1)$, by Remark \ref{rem:MaxCoordinate}, $DSV(w_0,j_2)(j_1)\leq w_1(j_1)$. So we have $DSV(w_0,j_1)\in S(w_1)$ and $w_0(j_1)\leq w_1(j_1).$ Then by Lemma \ref{lem:BoundedMovement}, it must be the case that $w_1=w_0$. 
\end{proof}

\section{Covered and Uncovered Vertices}\label{sec:covered}
We are trying to prove that for polynomial shape diagrams, for every ordering and all 
large enough $i$ there are no depth $i$ pairs (outside of a negligible 
exceptional set), that is, no paths $x,x'$ that have the same $i$ coding but different $i+1$ codings. 
If $x$ and $x'$ were a depth $i$ pair {then for some $m \in \Z$, $w=v_{i+1}(T^mx) \neq v_{i+1}
(T^mx')=w'$, and the source sets} $S(w)$ and $S(w')$ intersect. 
Thus we need to keep track of a key aspect of the interconnections among vertices in a diagram, namely which vertices at a given level have their source sets contained in those of other vertices at that level. This will allow precise tracking of vertices that the depth $i$ pair passes through.

\begin{definition}\label{def:covered}
{We say that a vertex $w$ is {\em covered} if there is a 
vertex $w' \neq w$ (at the same level as $w$) such that $S(w) 
\subseteq S(w')$, in which case we say that $w'$ {\em covers} $w$. 
Otherwise $w$ is said to be {\em uncovered}.} 
\end{definition}
Lemma \ref{lem:2DSV} implies that if a vector $w$ has at least two 
distinct coordinates, $j_1, j_2$, for which $DSV(w,j_1)$ and 
$DSV(w,j_2)$ are defined, then $w$ must be uncovered. While 
this is sufficient, it is not necessary. For example for 
$q\geq 3$, the vertex on level $2$ given by $w=(d,d-
1,1,0,\dots,0)$ is uncovered (see Proposition \ref{prop:covered}), but the only distinguished source
vertex for $w$ that exists is $DSV(w,1)$ 
{(we need to subtract exactly $d$ from a single entry of $w$ to get a $DSV$)}. The following 
discussion gives necessary and sufficient conditions for a 
vertex to be covered and in the case that it is covered, explicitly describes the 
vertices by which it is covered.

\begin{proposition}\label{prop:covered}
A vertex $w$ at a level {$n>q \geq 2$} is covered if and only if  there is a $j \in \{1, \dots, q\}$ such that $w(j)> (n-1)d$.
\end{proposition}
\begin{proof}
		We will prove the first direction by proving the contrapositive. Suppose that $w(i) \leq (n-1)d$ for all $i \in \{1, \dots, q\}$.
		 Because $n>q$, by Remark \ref{rem:MaxCoordinate}  there is at least one $j$ with $w(j) \geq d$. 
			Since $\sum_{i=1}^qw(i)=nd$ and $n>q\geq 2$,
		\begin{equation}
				 \sum_{i\neq j}w(i) \geq d.
				 \end{equation}
				
	We make two vertices $u$ and $u'$ in the source 
    {set} of $w$ that cannot both be in the source set of any $w' 
    \neq w$, as follows. Define
				\be
				u(i)=\begin{cases}
				       w(j) -d &\quad \text { if } i=j\\
				      w(i) &\quad\text{ if } i \neq j
				      \end{cases}
				      \en
				      {(so that $u = DSV(w,j)$).}
				      Note that $\sum_iu(i)=(n-1)d$. 
				      
				      For each $i \neq j$ choose {$n(i) \in [0,w(i)]$} so that $\sum_{i \neq j} n(i)= \sum_{i \neq j} w(i)-d$, and let
				      \be
				      u'(i)=\begin{cases}
				      	w(j) &\quad\text{ if } i =j\\
				      	n(i) &\quad\text{ if } i \neq j.
				      	\end{cases}
				      \en
	The idea is that to get $u'$ from $w$ we subtract, in any 
 possible way, a total of $d$ from coordinates not equal to $j$. Hence $\sum_{i}u'(i)=(n-1)d$.
	
By Lemma \ref{lem:BoundedMovement} for every $w'$ such that 
$u=DSV(w,j)\in S(w')$ and $w'\neq w$, we have $w'(j)<w(j)$. 
If $s'=w'-u'$ is a nonnegative vector with $|s'|=d$, and 
		$s'(j) \geq 0$ then
		 $w'(j)  \geq w(j)$. 
		So it is not possible that there is a $w'$ with $u$ and $u'$ both in $S(w')$. 
		Therefore such a vertex $w$ is uncovered.

To prove the converse, suppose that there is a $j \in \{1, 
\dots,q\}$ such that $w(j) > (n-1)d$; 
we will prove that then $w$ is covered.
 First, {we have}
 $\sum_{i \neq j} w(i) < d$,
 so that
 \be
c(w)=d-\sum_{i \neq j} w(i) \in [1,d]. 
 \en


 Choose $b \in [1,c(w)]$, let $\sigma$ be an integer vector with $q$ components such that $\sigma(j)=-b$, choose $\sigma(i) 
 \geq 0$ for $i \neq j$ such that $\sum_{i \neq j} \sigma (i) 
 =b$, and let $w'=w+\sigma$. We claim that $w'$ covers $w$.

  Suppose that
 	$u \in S(w)$. Then $u=w-s$ for some $s$; and $s(j) \geq c(w)$, since
 	 $s(j)<c(w)$ implies $|s|<c(w)+\sum_{i \neq j}s(i)\leq c(w) + \sum_{i \neq j}w(i)=d$.
  Therefore such a {triple $w, s, j$} satisfies the following useful condition.
 	\begin{condition}\label{cond:s}
 		The vertex $w \in \mathcal V_n$, nonnegative integer vector $s$ with $|s|=d$, and $j\in \{1, \dots, q\}$ satisfy
 			$w(j)>(n-1)d$, 
 			 $s(i) \leq w(i)$ for all $i$, and $s(j)\in [c(w),d]$.
 	\end{condition}
 	(Conversely, for any $w, j,$ and $s$ that satisfy Condition \ref{cond:s}, $u=w-s  \in S(w)$. In fact the converse holds even if we no longer assume $s(j)\in [c(w),d]$.)
  
We know that $1 \leq b \leq c(w) \leq s(j)$, so $s(j) \geq b$.
 Therefore $s' = \sigma + s \geq 0$ 
 and $u+ s' = w'$, showing that $u \in S(w')$. 
 Hence $S(w) \subseteq S(w')$, and $w$ is covered. 
  \end{proof}

\begin{remark}\label{rem:uncovered}The previous proposition is 
stated in terms of the definition of covered. We will 
frequently use it to determine that a vertex is uncovered. In 
particular, Proposition \ref{prop:covered} could be stated as 
follows: for $n>q\geq 2$, a vertex $w\in \mathcal{V}_n$ is uncovered if 
and only if for every $l\in \{1,\dots,q\}$ we have $w(\ell)\leq 
(n-1)d$ or equivalently if there exists $j$ for which $d\leq w(j)\leq (n-1)d$.\end{remark}

 	 Continue to assume through the end of Remark \ref{rem:cov2}
 	 	that $n>q\geq 2$.
 	 Now we will specify, given a covered vertex $w$, exactly which vertices $w' \neq w$ cover $w$.
 	 Since this information is not used in our main argument, the proof is omitted.
 	\begin{proposition}\label{prop:uncovered}
 		Suppose that $n > q \geq 2$. 
 		 Let $w,w'\in \mathcal{V}_n$ with $w\neq w'$. Assume that $j \in \{1, \dots, q\}$ satisfies $w(j)>(n-1)d$, so that $w$ is covered.
 		  Let $\sigma =w'-w$ (so that $\Sigma \sigma(i)=0$). 
 			Then $w'$ covers $w$ if and only if $\sigma$ satisfies the following equivalent conditions with respect to $w$:
 				\begin{condition}\label{cond:sum}
 				For each choice of $s$ satisfying Condition \ref{cond:s} with respect to $w$ and $j$, we have $s \neq \sigma + s \geq 0$. 
 			\end{condition}
 			\begin{condition}\label{cond:sig}
 				{Since $w \in \mathcal V_n$ and $j \in \{1, \dots,q\}$ satisfy $w(j)>(n-1)d$, defining $b=-\sigma(j)\in [1,c(w)]$, we have $\sigma(i) \geq 0$ for all $i \neq j$, and $\sum_{i \neq j} \sigma(i) =b$.}
 			\end{condition}
 			Moreover, $w'=w+\sigma$ covers $w$ and is itself uncovered if and only if $\sigma$ satisfies \ref{cond:sig} with $b=c(w)$.
 			\end{proposition}
{\begin{remark}\label{rem:cov2}
		The central part of
Proposition \ref{prop:uncovered} could be restated as: If there exists a $j\in \{1,2,\dots,q\}$ such that $w(j)>(n-1)d$, then $w'$ covers $w$ if and only if for $\sigma=w-w'$ we have $-d\leq \sigma(j)<0$, and for $i\neq j$, $\sigma(j)\geq 0$.
\end{remark}}

We mention a few more useful observations about uncovered vertices.
By Remark \ref{rem:MaxCoordinate}, every vertex in $\mathcal{V}_n$ has at 
least one coordinate greater than or equal to $nd/q$. 
Therefore, the hypothesis of the following lemma, (Lemma 
\ref{lem:SourceUncovered}) can be satisfied only if $nd/q\leq 
(n-2)d$, which implies that $2q/(q-1)\leq n$. See Figure \ref{fig:SourceUncovered}.

\begin{figure}
\begin{tikzpicture}[scale=1]
\foreach \n in {1,2}{
	\fill(0,-\n) circle (2pt);
	\fill(0,-\n+1) circle (2pt);
	\draw(0,-\n)--(-2,-\n-1);
	\draw(0,-\n)--(-1,-\n-1);
	\draw(0,-\n)--(0,-\n-1);
	\draw(0,-\n)--(1,-\n-1);
	\draw(0,-\n)--(2,-\n-1);
	\draw  (0,-\n) to [out=270,in=30] (0-1,-\n-1);
	\draw(0,-\n+1)--(-2,-\n);
	\draw(0,-\n+1)--(-1,-\n);
	\draw(0,-\n+1)--(0,-\n);
	\draw(0,-\n+1)--(1,-\n);
	\draw(0,-\n+1)--(2,-\n);
	\draw  (0,-\n+1) to [out=270,in=30] (-1,-\n);
	\draw (0,-\n) to [out=-20,in=95] (1,-\n-1);
	\draw (0,-\n) to [out=-20,in=95] (1,-\n-1);
	\draw  (0,-\n) to [out=270,in=160] (1,-\n-1);
	\draw  (0,-\n) to [out=270,in=160] (1,-\n-1);
	\draw (0,-\n+1) to [out=-20,in=95] (1,-\n);
	\draw (0,-\n+1) to [out=-20,in=95] (1,-\n);
	\draw  (0,-\n+1) to [out=270,in=160] (1,-\n);
	\draw  (0,-\n+1) to [out=270,in=160] (1,-\n);
	\pgfmathsetmacro{\p}{2*\n}
	\foreach \k in {1,2,...,\p}{
		\fill(\k,-\n) circle (2pt);
		\fill(-\k,-\n) circle (2pt);
		\draw(\k,-\n)--(\k-2,-\n-1);
		\draw(-\k,-\n)--(-\k-2,-\n-1);
		\draw(-\k,-\n) to [out=270,in=30] (-\k-1,-\n-1);
		\draw (\k,-\n) to [out=270,in=30] (\k-1,-\n-1);
		\draw (-\k,-\n) to [out=-20,in=95] (-\k+1,-\n-1);
		\draw (\k,-\n) to [out=-20,in=95] (\k+1,-\n-1);
		\draw  (-\k,-\n) to [out=270,in=160] (-\k+1,-\n-1);
		\draw  (\k,-\n) to [out=270,in=160] (\k+1,-\n-1);
		\draw(\k,-\n)--(\k-1,-\n-1);
		\draw(-\k,-\n)--(-\k-1,-\n-1);
		\draw(\k,-\n)--(\k,-\n-1);
		\draw(-\k,-\n)--(-\k,-\n-1);
		\draw(\k,-\n)--(\k+1,-\n-1);
		\draw(-\k,-\n)--(-\k+1,-\n-1);
		\draw(\k,-\n)--(\k+2,-\n-1);
		\draw(-\k,-\n)--(-\k+2,-\n-1);
  }}
    \foreach \k in {-6,-5,...,6}{
    \fill(\k,-3) circle (2pt);}
    \fill(0,-4) circle (2pt);
    \node at (0,-4.5){$(8,8)$};
    \foreach \k in {-2,-1,...,2}{
    \draw (0,-4)--(\k,-3);}
    
		\draw(-1,-3) to [out=270,in=160] (0,-4);
		\draw (-1,-3) to [out=-20,in=95] (0,-4);
  \draw(1,-3) to [out=270,in=30] (0,-4);
\end{tikzpicture}\caption{Level 4 is the first level to have a vertex whose entire source set is uncovered. In order for a vertex source set to be entirely uncovered, all coordinates should be less than or equal to $(n-2)d$ which in this case is $2(4)=8$.}
\label{fig:SourceUncovered}
\end{figure}
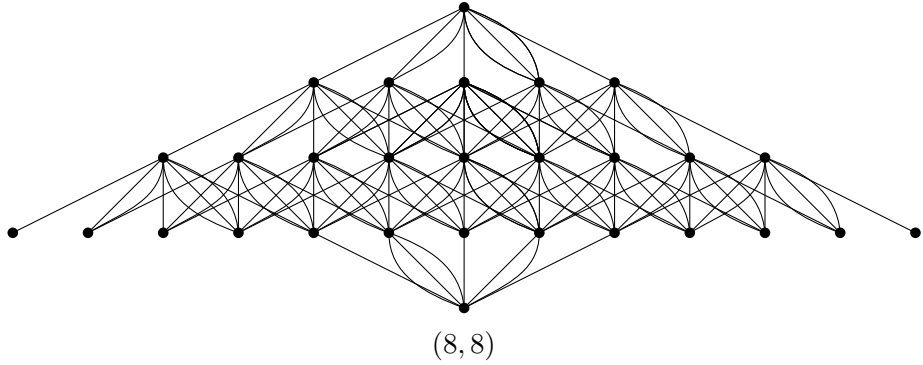
\begin{lemma}\label{lem:SourceUncovered}
    For $n>q\geq 2$, if $w\in \mathcal{V}_n$ is such that for every $l\in \{1,\dots,q\}$, $w(\ell)\leq (n-2)d$, then every vertex in $S(w)$ is uncovered.
\end{lemma}
\begin{proof}
    Let $u\in S(w)$. By Remark \ref{rem:MaxCoordinate}, for all $l\in \{1,\dots,q\}$, $u(\ell)\leq w(\ell)$. Since $u\in \mathcal{V}_{n-1}$, by Remark \ref{rem:uncovered}, $u$ is uncovered since for every $\ell \in\{i...q\}$.  $u(\ell)\leq (n-2)d$. 
\end{proof}
\begin{corollary}\label{cor:SourceUncovered}
    Given $w\in \mathcal{V}_n$, if there exists $j \in \{1,\dots,q\}$ such that $2d\leq w(j)\leq (n-2)d$, then every vertex in $S(w)$ is uncovered.
\end{corollary}
\begin{proof}
   Since $2d\leq w(j)$, we have that \[2d+\sum_{l\neq j}w(\ell)\leq w(j)+\sum_{l\neq j}w(\ell)=nd.\]
   Hence, $\sum_{l\neq j}w(\ell)\leq (n-2)d$. 
   Therefore, for each $l\neq j$ we have that $w(\ell)\leq (n-2)d$ and by assumption, 
   $w(j)\leq (n-2)d$. Hence, for all $\ell \in\{1,\dots,q\}$, $w(\ell)\leq (n-2)d$. 
   Then by Lemma \ref{lem:SourceUncovered}, every vertex in $S(w)$ is 
   uncovered.
\end{proof}

\begin{figure}
\begin{tikzpicture}[scale=1]
\foreach \n in {1,2}{
	\fill(0,-\n) circle (2pt);
	\fill(0,-\n+1) circle (2pt);
	\draw(0,-\n)--(-2,-\n-1);
	\draw(0,-\n)--(-1,-\n-1);
	\draw(0,-\n)--(0,-\n-1);
	\draw(0,-\n)--(1,-\n-1);
	\draw(0,-\n)--(2,-\n-1);
	\draw  (0,-\n) to [out=270,in=30] (0-1,-\n-1);
	\draw(0,-\n+1)--(-2,-\n);
	\draw(0,-\n+1)--(-1,-\n);
	\draw(0,-\n+1)--(0,-\n);
	\draw(0,-\n+1)--(1,-\n);
	\draw(0,-\n+1)--(2,-\n);
	\draw  (0,-\n+1) to [out=270,in=30] (-1,-\n);
	\draw (0,-\n) to [out=-20,in=95] (1,-\n-1);
	\draw (0,-\n) to [out=-20,in=95] (1,-\n-1);
	\draw  (0,-\n) to [out=270,in=160] (1,-\n-1);
	\draw  (0,-\n) to [out=270,in=160] (1,-\n-1);
	\draw (0,-\n+1) to [out=-20,in=95] (1,-\n);
	\draw (0,-\n+1) to [out=-20,in=95] (1,-\n);
	\draw  (0,-\n+1) to [out=270,in=160] (1,-\n);
	\draw  (0,-\n+1) to [out=270,in=160] (1,-\n);
	\pgfmathsetmacro{\p}{2*\n}
	\foreach \k in {1,2,...,\p}{
		\fill(\k,-\n) circle (2pt);
		\fill(-\k,-\n) circle (2pt);
		\draw(\k,-\n)--(\k-2,-\n-1);
		\draw(-\k,-\n)--(-\k-2,-\n-1);
		\draw(-\k,-\n) to [out=270,in=30] (-\k-1,-\n-1);
		\draw (\k,-\n) to [out=270,in=30] (\k-1,-\n-1);
		\draw (-\k,-\n) to [out=-20,in=95] (-\k+1,-\n-1);
		\draw (\k,-\n) to [out=-20,in=95] (\k+1,-\n-1);
		\draw  (-\k,-\n) to [out=270,in=160] (-\k+1,-\n-1);
		\draw  (\k,-\n) to [out=270,in=160] (\k+1,-\n-1);
		\draw(\k,-\n)--(\k-1,-\n-1);
		\draw(-\k,-\n)--(-\k-1,-\n-1);
		\draw(\k,-\n)--(\k,-\n-1);
		\draw(-\k,-\n)--(-\k,-\n-1);
		\draw(\k,-\n)--(\k+1,-\n-1);
		\draw(-\k,-\n)--(-\k+1,-\n-1);
		\draw(\k,-\n)--(\k+2,-\n-1);
		\draw(-\k,-\n)--(-\k+2,-\n-1);
  }}
    \foreach \k in {-6,-5,...,6}{
    \fill(\k,-3) circle (2pt);}
    \fill[red](0,-2) circle (3pt);
    \fill[blue](-2,-3) circle (3pt);
    \fill[blue](-1,-3) circle (3pt);
    \fill[blue](0,-3) circle (3pt);
    \fill[blue](1,-3) circle (3pt);
    \fill[blue](2,-3) circle (3pt);
\end{tikzpicture}\caption{Vertex $(4,4)$, pictured in red, in level $2$ is uncovered, and every vertex in level $3$ with $(4,4)$ its source set, pictured in blue, is also uncovered.}
\label{fig:TargetUncovered}
\end{figure}
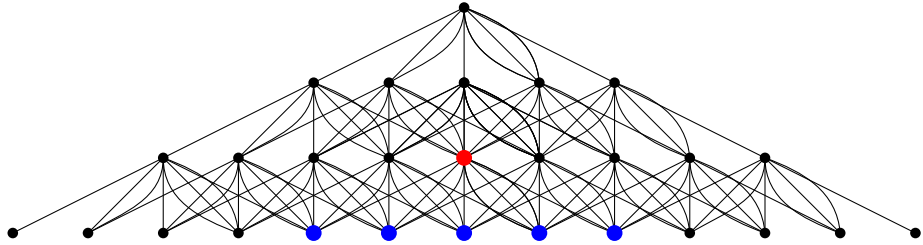

\begin{lemma}\label{lem:TargetUncovered}
   For $n-1>q\geq 2$, if $u\in \mathcal{V}_{n-1}$ is uncovered, and $u\in S(w)$, then $w$ is uncovered. 
\end{lemma}
\begin{proof}
 Assume $u\in \mathcal{V}_{n-1}$ is uncovered. By Remark \ref{rem:uncovered}, for all $l\in \{1,\dots,q\}$, $u(\ell)\leq (n-2)d$. Further, for all $\ell \in\{1,\dots,q\}$, $w(\ell)\leq u(\ell)+d$ (by Remark \ref{rem:MaxCoordinate}). 
 So, $w(\ell)\leq u(\ell)+d\leq (n-2)d+d=(n-1)d$. Then by Remark \ref{rem:uncovered}, $w$ is uncovered. See Figure \ref{fig:TargetUncovered}.
\end{proof}

\section{Chains and Links}\label{sec:chains}

In this section we develop some of the machinery that will be used in the main theorem,
continuing to work in a fixed polynomial shape diagram with path space $X$. 
\begin{definition}\label{def:chain}
		Given $k, n\geq 1$, a \emph{chain of length $k$ at level $n$} is a double sequence
		\begin{equation}
		w_0 \overset{u_0}{\longleftrightarrow} w_1 \overset{u_1}{\longleftrightarrow} w_2 \cdots w_{k-1} \overset{u_{k-1}}{\longleftrightarrow} w_k
		\end{equation}
		of vertices $w_0, \dots, w_k$ at level $n$, called {\em splitting vertices}, 
		and $u_0, \dots, u_{k-1}$ at level $n-1$, called {\em shared vertices}, 
		such that for every $0 \leq \ell < k$, $u_{\ell} \in S(w_{\ell}) \cap S(w_{\ell+1})$. 
		We say that the chain {\em links the two vertices $w_0$ and $w_k$.}
		The chain is \emph{straight} if all vertices are distinct.
		A chain of length $1$ is called a {\em link}.
	\end{definition}

     \begin{definition}\label{def:dc}
        We say that a straight chain \bee
	w_0 \overset{u_0}{\longleftrightarrow} w_1 \overset{u_1}{\longleftrightarrow} w_2\overset{u_2}{\longleftrightarrow} w_3 \overset{u_3}{\longleftrightarrow} w_4 \overset{u_4}{\longleftrightarrow}... \overset{u_{k-1}}{\longleftrightarrow} w_{k}\ene is a \emph{distinguished chain in direction $j$}, if for each $1\leq \ell< k$, $u_\ell$ is $DSV(w_{\ell},j)$.
    \end{definition}

     We do not want to impose in the above definition any restrictions on $u_0$ because our ultimate goal is to extend any arbitrary link to a distinguished chain.  
     	The following definition will be applied to a depth $i$ pair to find sequences of vertices at levels $i$ and $i+1$ that eventually cause a contradiction.
    
\begin{definition}\label{def:xchain}
		For $i \geq 1$ and a pair of paths $x,x'$, a chain 
		\bee
		w_0 \overset{u_0}{\longleftrightarrow} w_1 \overset{u_1}{\longleftrightarrow} w_2 \cdots w_{k-1} \overset{u_{k-1}}{\longleftrightarrow} w_k
		\ene
		of length $k$ at level $i+1$ is called an {\em $x,x'$ chain} if 
			there is a strictly monotonic sequence of integers $m_0=0,m_1, \dots , m_{k-1}$, called \emph{the switching times}, such that \\
			for each $1 \leq \ell \leq  k$, $u_{\ell}=v_{i}(T^{m_{\ell}}x)=v_i(T^{m_{\ell}}x')$;\\
			$w_0=v_{i+1}(x)$, $w_1=v_{i+1}(x')$;\\
				for odd $\ell \in [1,k)$, $w_{\ell}=v_{i+1}(T^{m_{\ell-1}}x')=v_{i+1}(T^{m_{\ell}}x')$;\\
				{for even $\ell \in [2,k)$}, $w_{\ell}=v_{i+1}(T^{m_{\ell-1}}x)=v_{i+1}(T^{m_{\ell}}x)$;\\
				and $w_k=v_{i+1}(T^{m_{k-1}}x)$ if $k$ is even, $w_k=v_{i+1}(T^{m_{k-1}}x')$ if $k$ is odd. \\
				An $x,x'$ chain of length $1$ is called an {\em $x,x'$ link}.
			See Figure \ref{fig:xx'chain}. 
	\end{definition}

	\begin{figure}[h]
		\begin{tikzpicture}[scale=1.5]
			\fill(0,0) circle (2pt);
			\fill(1,1) circle (2pt);
			\fill(2,0) circle (2pt);
			\fill(3,1) circle (2pt);
			\fill(4,0) circle (2pt);
			\fill(5,1) circle (2pt);
			\fill(6,1) circle (2pt);
			\fill(7,0) circle (2pt);
			\fill(8,1) circle (2pt);
			\fill(9,0) circle (2pt);
			\draw(0,0)--(1,1)--(2,0)--(3,1)--(4,0)--(5,1);
			\draw(6,1)--(7,0)--(8,1)--(9,0);
			\node at (0,-.2){$w_0$};
			\node at (2,-.2){$w_1$};
			\node at (4,-.2){$w_2$};
			\node at (1,1.2){$u_0$};
			\node at (3,1.2){$u_1$};
			\node at (5,1.2){$u_2$};
			\node at (6,1.2){$u_{k-2}$};
			\node at (7,-.2){$w_{k-1}$};
			\node at (8,1.2){$u_{k-1}$};
			\node at (9,-.2){$w_k$};
			\node[fill=white] at (.25,.5){$x_i$};
			\node at (5.5,1){$\dots\dots$};
			\node[fill=white] at (1.75,.4){$x'_i$};
			\node[fill=white] at (2.77,.3){{$(T^{m_1}x')_i$}};
			\node[fill=white] at (3.65,.8){{$(T^{m_1}x)_i$}};
			\node[fill=white] at (4.77,.3){{$(T^{m_2}x)_i$}};
			\node[fill=white] at (6.7,.8){{$(T^{m_{k-2}}x')_i$}};
			\node[fill=white] at (7.77,.3){{$(T^{m_{k-1}}x')_i$}};
			\node[fill=white] at (8.7,.8){{$(T^{m_{k-1}}x)_i$}};
		\end{tikzpicture}
		\caption{An example of an $x,x'$ chain of length $k$ where $k$ is even. Other vertices are allowed, but those pictured are the ones that make up the chain.}
		\label{fig:xx'chain}
	\end{figure}

 \begin{lemma}\label{lem:LinkLemma}[Link Lemma]
	For $i>q\geq 2$, if $x,x'$ is a depth $i$ pair of paths, then
	every straight $x,x'$ link 
	$w_0 \overset{u_0}{\longleftrightarrow} w_1$  with $v_{i+1}(x)=w_0$, $v_{i+1}(x')=w_1$, such that 
 there exists a $j\in \{1,\dots,q\}$ for which $2d\leq w_1(j)\leq w_0(j)\leq (i-1)d$, extends to a straight 
 $x,x'$ chain of length $2$
 \[w_0 \overset{u_0}{\longleftrightarrow} w_1\overset{DSV(w_1,j)}{\longleftrightarrow}w_2,\]
 with $m_1$, as in the definition of an $x,x'$ chain, satisfying $v_{i+1}(T^{m_1}x')=w_1$, $v_{i+1}(T^{m_1}x)=w_2$, and moreover for all $m\in [0,m_1]$, $v_{i+1}(T^mx')=w_1$, and $v_{i+1}(T^mx)\in \{w_0,w_2\}$.
\end{lemma}
\begin{proof}
By assumption, there exists a $j$ such that $2d\leq w_1(j)\leq 
w_0(j)\leq (i-1)d$, which by Remark \ref{rem:uncovered} is more than sufficient to ensure that $w_0$ and 
$w_1$ are both uncovered. 
Since $w_1(j)\leq w_0(j)$, by Lemma \ref{lem:BoundedMovement}
	$DSV(w_1,j) \notin S(w_0)$. In particular, $DSV(w_1,j) \neq u_0 =v_i(x')$. 
Choose $m_1 \neq 0$ with $|m_1|$ minimal such that 
$v_i(T^{m_1}x')=DSV(w_1,j)$, and for all $m\in 
[0,m_1]$ (or $[m_1,0]$ if $m_1<0$), $v_{i+1}(T^mx')=w_1$. 
Then define 
$w_2=v_{i+1}(T^{m_1}x)$. Replacing $T$ by $T^{-1}$ if necessary, assume $m_1>0$. Since $x,x'$ are depth 
$i$, $v_i(T^{m_1}x)=DSV(w_1,j)$ and $m_1$ is also the first time that the forward orbit of $x$ hits $DSV(w_1,j)$.

Since $DSV(w_1,j) \notin S(w_0)$, $w_2 \neq w_0$.
We claim that also $w_2 \neq 
w_1$. Suppose to the contrary, that $w_2=w_1$. By the way we chose 
$m_1$, $x'$ is still at $w_1$ at time $m_1$. Look at the edge numbers 
based at $w_1$.  
The orbit of $x$ has entered $w_1$ before or at time $m_1$, but after 
$x'$ (which is at $w_1$ at time 
$0$), so $\xi((T^{m_1}x)_{i}) < \xi((T^{m_1}x')_{i})$. 
In other words, $(T^{m_1}x)_i$ connects $w_1$ and $DSV(w_1,j)$, as does $(T^{m_1}x')_i$, 
and $(T^{m_1}x)_i$ has smaller edge label than $(T^{m_1}x')_i$.
				
Going backwards in time, the orbit of $x'$ has to hit $DSV(w_1,j)$ again. Hence it passes through edge $(T^{m_1}x)_i$ at some time $m_2<m_1$. Specifically, the backward orbit of $x'$ goes through edges based at $w_1$, starting 
at edge number $\xi((T^{m_1}x')_{i})$ and working down to edge number 
$\xi((T^{m_1}x)_{i})$. It stays at $w_1$ all this time, at least 
through time $m_2$.  By the minimality of $m_1$, there is no time $m$ in 
$[0,m_1)$ such that $v_i(T^mx')=DSV(w_1,j)$. Therefore, it 
is necessarily the case that $m_2<0$. For all $m\in 
[m_2,m_1]$, since $v_{i+1}(T^mx')=w_1$, we have $v_i(T^mx')=v_i(T^mx)\in 
S(w_1)$. Further, since $DSV(w_1,j)\notin S(w_0)$, we know that $v_i(T^{m_2}x)\neq w_0$. 
Therefore we have that $v_i(T^{m_2}x)\neq w_0, v_i(x)=w_0$, and 
$v_i(T^{m_1}x)\neq w_0$. Then by the definition of the Vershik map, 
the orbit of $x$ must pass through every vertex in $S(w_0)$ between 
time $m_2$ and $m_1$. Hence it is necessarily the case that 
$S(w_0)\subseteq S(w_1)$ and $w_1$ covers $w_0$. This contradicts our assumption that $w_0$ 
is uncovered, proving that $w_2\neq w_1$ and
			\begin{equation}
		w_0 \overset{u_1}{\longleftrightarrow} w_1 \overset{DSV(w_1,j)}{\longleftrightarrow} w_2 
		\end{equation}
		is a straight $x,x'$ chain of length $2$.

  We claim that between times $0$ and $m_1$, the orbit of 
 $x$ cannot pass through any vertices at level $i+1$ other than 
 $w_0$ and $w_2$. For all 
 $m$ in  $[0,m_1]$, $v_{i+1}(T^mx')=w_{1}$ and hence 
 $v_{i}(T^mx')\in S(w_{1})$; further, since $x$ and $x'$ are a depth 
 $i$ pair, $v_i(T^mx)=v_i(T^mx')\in S(w_1)$. Therefore, if the orbit of 
 $x$ leaves $w_{0}$ and passes through some vertex $w\neq 
 w_{2}$ between time $0$ and $m_{1}$, then $w$ 
 must be covered by $w_{1}$. 
 
 However, since $2d\leq w_{1}(j)\leq 
 (i-1)d$, by Corollary \ref{cor:SourceUncovered} every vertex in $S(w_{1})$ is uncovered. 
 In particular, for each $m\in [0, m_1]$, $v_{i}(T^mx)$ is 
 uncovered. Then by Lemma \ref{lem:TargetUncovered}, $v_{i+1}(T^mx)$ 
 is also uncovered. In particular, $v_{i+1}(T^mx)$ cannot be covered 
 by $w_{1}$. Therefore between time $0$ and $m_{1}$, the 
 orbit of $x$ can only pass through the vertices $w_{0}$ 
 and $w_{2}$.
	\end{proof}

\section{{Inherent Expansiveness of Polynomial Shape Diagrams}}\label{sec:main}
$X$ continues to denote the path space of a fixed polynomial shape diagram (Definition \ref{def:polyshape}).
We intend to prove that outside of the negligible exceptional set 
defined above (Definition \ref{def:genericpaths}) there are no 
pairs of paths of any large enough depth. The idea is to show that 
for any sufficiently large $i$ and any possible depth $i$ pair 
$x,x'$ in $X(\xi)$ (the set of paths which are not in the orbit of 
any maximal or minimal path and have dense biinfinite orbits), one 
can construct a distinguished $x,x'$ chain in which the splitting 
vertices are uncovered and the orbits of $x$ and $x'$ meet 
vertices of the chain in an alternating manner. The chain will 
have sufficient length ($2d+3$) to ensure that when we look at 
level $i+2$, we find that the {alternation} contradicts the 
definition of the Vershik map. To construct this chain, it will be 
necessary to find two uncovered vertices forming a link (Lemma 
\ref{lem:GettingToTheGivenMV}), each with $j$'th component of size 
at least $2d^2+4d$ for some $j\in \{1,\dots,q\}$, with which to 
start the chain. The following Lemma 
allows us to find such a link if we look low enough in the 
diagram.

\begin{lemma}[Chain Starting]\label{lem:GettingToTheGivenMV}
    Let an ordering $\xi$ be given, $N= 4dq+6q$, and $i>N$. 
    For any depth $i$ pair $x,x'\in X(\xi)$
    there exists a time in their orbits when at level $i+1$ 
    they pass through distinct uncovered vertices $v$ and $v'$ 
    for which there is a $j\in \{1,2,\dots,q\}$ such that 
    \begin{equation}\label{eqn:ChainStart}
    2d^2+4d\leq v'(j)<v(j)\leq (i-1)d<id.\end{equation} 
\end{lemma}

\begin{proof}
    Let $\xi$ be given, $i>N$ and suppose that $x,x'\in X(\xi)$ is a depth $i$ pair. Then by Proposition \ref{prop:distinct}
    there is a time in 
    their orbits such that $x$ and $x'$ pass through different 
    vertices at level $i+1$, say $w_0$ and $w_0'$, respectively. 
    Without loss of generality, assume this happens at time $0$. 
   
   Since $w_0,w_0'\in \mathcal{V}_{i+1}$ and $i+1\geq 
   4dq+6q$, Remark \ref{rem:MaxCoordinate} tells us that they each 
   have a component 
   (not necessarily the same one) of size at least
  \[(i+1)d/q\geq (4dq+6q)d/q=4d^2+6d.\] 

  Let $j\in \{1,\dots,q\}$ such that $w_0'(j)\geq 4d^2+6d$. Since $v_{i}(x)=v_{i}(x')$, $w_0$ and $w_0'$ share a common 
  source vertex and hence each of their components differ by at most 
  $d$. (See Remark \ref{rem:MaxCoordinate}.) Therefore, the $j$'th components of {\em both} vertices are of size at least $4d^2+5d$. Switching names if necessary, assume $4d^2+5d\leq w_0(j)\leq w_0'(j)$. 
   
We do not know if our initial $w_0$ and $w'_0$ meet the conditions 
of (\ref{eqn:ChainStart}) or not. However, the following procedure 
will produce two new vertices that meet the conditions of 
(\ref{eqn:ChainStart}), regardless whether $w_0$ and $w'_0$ do.  We will produce an $x,x'$ chain of length $2d+2$ starting with the link between $w'_0$ and $w_0$ for 
which the $j$'th components of the splitting vertices are 
decreasing by at least one and by at most $d$ each time. Since we know that $4d^2+5d\leq w_0(j)\leq (i+1)d$, the splitting vertex $w_{2d+1}$ at the end of the 
chain must have $j$'th component at most $(i+1)d-(2d+1)=(i-1)d-1<(i-1)d$ but at least $4d^2+5d-(2d+1)d=2d^2+4d>2d$ and hence be uncovered, and have all vertices in its source set uncovered by Corollary \ref{cor:SourceUncovered}. (The ultimate goal is to have the end of this chain start (in the next theorem) a new chain of adequate length for which every splitting vertex has an uncovered source set. The above mentioned vertex $w_{2d+1}$ having $j$'th component $2d^2+4d$ {accomplishes this.)}

    Switching to $T^{-1}$ if necessary, let $m_1>0$ be the first 
    time in the orbit of $x$ such that $v_i(T^{m_1}x)=DSV(w_0,j)$ 
    and for all $m$ in $[0,m_1]$, $v_{i+1}(T^mx)=w_0$. First 
    consider the case in which $v_{i+1}(T^{m_1}x')=w_0$. Then the 
    orbit of $x'$ enters $w_0$ after $x$, so $\xi((T^{m_1}x')_{i})
    <\xi((T^{m_1}x)_{i})$. By density, the orbits of $x$ and $x'$ 
    must eventually leave $w_0$, but $\xi((T^{m_1}x')_{i})
    <\xi((T^{m_1}x)_{i})$ implies this must occur at different 
    times. Let $n_1> m_1$ be the first time that 
    $v_i(T^{n_1}x)=v_i(T^{n_1}x')=DSV(w_0,j)$ and $v_{i+1}
    (T^{n_1}x)\neq v_{i+1}(T^{n_1}x')$. In other words, $n_1$ is 
    the first time after one of these orbits (of $x$ and $x'$) 
    leave $w_0$ that both orbits meet the $DSV(w_0,j)$.

    Without loss of generality, assume  the orbit of $x'$ 
    leaves $w_0$ first. We claim that $v_{i+1}
    (T^{n_1}x)=w_0$. Consider the last time $n'_1<n_1$ 
    that  $v_i(T^{n'_1}x)=v_i(T^{n'_1}x')=DSV(w_0,j)$ and 
    $v_{i+1}(T^{n'_1}x)= v_{i+1}(T^{n'_1}x')=w_0$. Then 
    $\xi((T^{n'_1}x)_i)<\xi((T^{n'_1}x')_i)$ (since $x'$ leaves $w_0$ first) implies the forward 
    orbit of $T^{n'_1}x$ must still pass through $(T^{n_1'}x')_i$ before it can leave $w_0$. In 
    particular this happens after time $n'_1$. Since 
    $n_1$ is the first time after $n'_1$ that the 
    orbits of $x$ and $x'$ visit $DSV(w_0,j)$, it must be 
    the case that the orbit of $x$ is still at 
    $w_0$ at time $n_1$. So we have that $v_{i+1}(T^{n_1}x)=w_0$ and $v_{i+1}(T^{n_1}x')\neq w_0$. See Figure \ref{fig:chainstart}.

    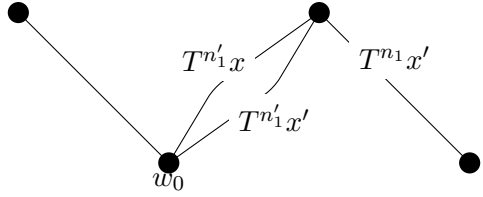
\begin{figure}
        \centering
        \begin{tikzpicture}[scale=2]
        \fill(0,1) circle (2pt);
        \fill(2,1) circle (2pt);
         \fill(1,0) circle (2pt);
        \fill(3,0) circle (2pt);
        \draw(0,1)--(1,0);
        \draw[rounded corners](1,0)--(1.3,.5)--(2,1);
        \draw[rounded corners](1,0)--(1.7,.5)--(2,1);
        \draw(2,1)--(3,0);
        \node[fill=white] at (1.3,.7){$T^{n'_1}x$};
        \node[fill=white] at (1.7,.3){$T^{n'_1}x'$};
        \node[fill=white] at (2.5,.7){$T^{n_1}x'$};
        \node[below] at (1,0){$w_0$};
        \end{tikzpicture}
        \caption{$\xi((T^{n_1'}x)_i)<\xi((T^{n_1'}x')_i)$ and so the orbit of $x$ must pass through $(T^{n_1'}x')_i$ at some time greater than or equal to time $n_1$ and hence is still at vertex $w_0$ at time $n_1$.  }
        \label{fig:chainstart}
    \end{figure}
    
       For the second case in which $v_{i+1}(T^{m_1}x')\neq w_0$, let $n_1=m_1$. Then in either case we have $v_i(T^{n_1}x)=v_i(T^{n_1}x')=DSV(w_0,j)$, $v_{i+1}(T^{n_1}x)=w_0$ and $v_{i+1}(T^{n_1}x')\neq w_0$. Call 
    $v_{i+1}(T^{n_1}x')=w_1$. Then $w_1\neq w_0$, but $DSV(w_0,j)\in S(w_1)$.
     So by Lemma \ref{lem:BoundedMovement} and  Remark \ref{rem:MaxCoordinate}, $1\leq w_0(j)-w_1(j)\leq d$.
     Furthermore, since $w_0'(j)\geq w_0(j)$, Lemma \ref{lem:BoundedMovement} also implies that $DSV(w_0,j)\notin S(w_0')$, indicating that $w_1 \neq w_0'$. 

     Since $S(w_1)\cap S(w_0)\neq \emptyset$, by Remark \ref{rem:MaxCoordinate}, $w_1(j)\geq 4d^2+4d>d$. Hence by Remark \ref{rem:CondForDSV}, $DSV(w_1,j)$ exists. Arguing as above, there exists a smallest $n_2>n_1$ such that 
  $v_i(T^{n_2}x)=v_i(T^{n_2}x')=DSV(w_1,j)$ and $v_{i+1}
  (T^{n_2}x)\neq v_{i+1}(T^{n_2}x')$. 
 Therefore, one of $v_{i+1}(T^{n_2}x)$ and $v_{i+1}
  (T^{n_2}x')$ must be $w_1$. Let the other be $w_2$. Then we have 
  $2\leq w_0(j)-w_2(j)\leq 2d$. Continue in this manner to define 
$w_3, w_2,\dots, w_{2d},w_{2d+1}$. See Figure 
  \ref{fig:ChainToUncovered}.

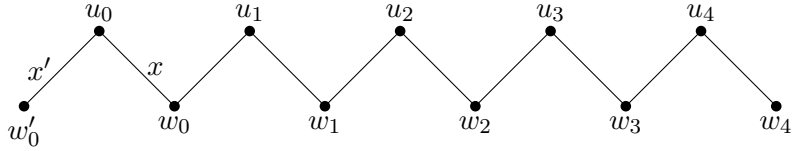
\begin{figure}
    \centering
    \begin{tikzpicture}
        \foreach \k in {0,2,4,6,8,10}{
        \fill(\k,0) circle (2pt);}
        \foreach \k in {1,3,5,7,9}{
        \fill(\k,1) circle (2pt);
        \draw(\k,1)--(-1+\k,0);
        \draw(\k,1)--(1+\k,0);}
        \node[below] at (0,0){$w_0'$};
        \node[below] at (2,0){$w_0$};
        \node[below] at (4,0){$w_1$};
        \node[below] at (6,0){$w_2$};
        \node[below] at (8,0){$w_3$};
        \node[below] at (10,0){$w_4$};
        \node[above] at (1,1){$u_0$};
        \node[above] at (3,1){$u_1$};
        \node[above] at (5,1){$u_2$};
        \node[above] at (7,1){$u_3$};
        \node[above] at (9,1){$u_4$};
        \node[left] at (.5,.5){$x'$};
        \node[right] at (1.5,.5){$x$};
    \end{tikzpicture}
    \caption{We build the chain where for $l\geq 1$, $u_l=v_i(T^{n_l}x)=v_i(T^{n_l}x')$ and $w_l$ is $v_{i+1}(T^{n_l}x)$ or $v_{i+1}(T^{n_l}x')$ and $w_{l-1}$ is the other vertex.}
    \label{fig:ChainToUncovered}
\end{figure}

As $\ell$ increases from $0$ to $2d$, the $j$'th coordinate 
of $w_\ell$ 
decreases by at 
least one each time and by at most $d$ each time. Therefore by the time 
we produce $w_{2d}$, the $j$'th coordinate has decreased by at least 
$2d$ from the $j$'th coordinate of $w_0$ ensuring that $w_{2d}(j)\leq w_0(j)-2d\leq (i+1)d-2d=(i-1) d$. 
Further, the $j$'th coordinate decreases by at most $2d^2$.

Then we have
 \[w_0(j)-w_{2d}(j)\leq 2d^2,   \]
        so that
                \[w_{2d}(j)\geq w_0(j)-2d^2\geq 4d^2+5d-2d^2=2d^2+5d.\]
{Also} $w_{2d}(j)-w_{2d+1}(j)\leq d$ implies that 
\[d\geq w_{2d}(j)-w_{2d+1}(j)\geq 2d^2+5d- w_{2d+1}(j), \]
and hence
\[w_{2d+1}(j)\geq 2d^2+5d-d=2d^2+4d. \]

  Further, since the $j$'th component is decreasing each time, $w_{2d}$ and $w_{2d+1}$ are distinct and $2d^2+4d\leq w_{2d+1}(j)<w_{2d}(j)\leq (i-1) 
  d<id$. Hence, by Remark \ref{rem:uncovered},   $v=w_{2d}$ 
  and $v'=w_{2d+1}$ are distinct uncovered vertices. 
  Therefore we have shown that the (dense) orbits of the 
  depth $i$ pair $x,x'$ pass through distinct uncovered vertices $v$ and $v'$ for which $2d^2+4d\leq v'(j)
  <v(j)\leq (i-1) 
  d<id$. 
\end{proof}

\begin{lemma}
    \label{lem:SourceSet}
    Let $z\in \mathcal{V}_{n+1}$ and $j\in \{1,2,\dots,q\}$ be such that $d\leq z(j)\leq nd$. Then there exist $w_0,w_1,\dots, w_d\in S(z)$ such that $w_{\ell}(j)=z(j)-\ell$ for all $\ell=0,1,\dots,d$. 
\end{lemma}

\begin{proof}
Since $d\leq z(j)\leq nd$ and $\sum_{i=1}^qz(i)=
(n+1)d$, we have $d\leq \sum_{i\neq j}z(i)\leq nd$. We can then find an ${s}_0\in S$ such that $s_0(j)=0$, $\sum_{i\neq 
j}s_0(i)=d$ and $z-s_0$ has all nonnegative coordinates. Hence 
$w_0=z-s_0\in \mathcal{V}_n$. 
This means $w_0\in S(z)$ and $w_0(j)=z(j)$. We will proceed in a recursive 
manner. Since $\sum_{i\neq j}s_0(i)=d$, there exists a $j_1\neq 
j$ such that $s_0(j_1)\geq 1$. 
Define $s_1=s_0-e_{j_1}+e_j$ 
(recalling that $e_{j_1}$ and $e_j$ are standard basis vectors in 
$q$-space) and let  ${w}_1={z}-{s_1}$. 
Then $w_1$ has all 
coordinates nonnegative and hence is in $\mathcal{V}_n$ and $S(z)$. 
Further, $\sum_{i\neq j}s_1(i)=d-1$, $s_1(j)=1$, and $w_1(j)=z(j)-s_1(j)=z(j)-1$. We continue recursively, producing 
$s_{\ell}\in S(z)$ from $s_{\ell-1}$, for each $\ell \leq d$.  
Specifically, given that $\sum_{i\neq j}s_{\ell-1}(i)=d-(\ell-1)$ and $s_{\ell -1}(j)=\ell-1$ for some $\ell \leq d$, choose $j_{\ell}\neq j$ such that $s_{\ell-1}(j_{\ell})\geq 1$ 
and define $s_{\ell}=s_{\ell-1}-e_{j_{\ell}}+e_j$. Then define 
$w_{\ell}=z-s_{\ell}$. So, $s_{\ell}(j)=\ell$, $\sum_{i\neq j}s_{\ell}(i)=d-\ell$, and $w_{\ell} \geq 0$. Hence $w_{\ell} \in S(z)$ and 
$w_{\ell}(j)=z(j)-\ell$. 
\end{proof}

Recall that, given an ordering $\xi$, $X(\xi)$ denotes the set of paths which are not in the orbit of any maximal or minimal path and have a dense (biinfinite) orbit. By Propositions \ref{prop:minmax} and \ref{prop:measures} (and Remark \ref{rem:measures}), $X(\xi)$ is comeager and has full measure with respect to every fully supported ergodic invariant measure on $X$.

\begin{theorem}\label{thm:main}
  Suppose we have an unordered polynomial shape Bratteli diagram $\mathcal B=(\mathcal V,\mathcal E)$, with path space $X$. Then for any ordering $\xi$ of the diagram, there exists $N$ such that for each $i\geq N$ there is no depth $i$ pair of paths in $X(\xi)$. That is, the Bratteli diagram defining $X$ is inherently expansive: for every ordering $\xi$ of $X$ the Vershik map is expansive on $X(\xi)$.
\end{theorem}

\begin{proof}
    Let $\xi$ be an ordering of $\mathcal B$ and choose $N=4dq+6q$.
    Suppose to the contrary that for some $i\geq N$ there 
    exists a depth $i$ pair $x, x'$ in $X(\xi)$. 
    Then by Lemma \ref{lem:GettingToTheGivenMV}, there exists a 
    time when the orbits of $x$ and $x'$ pass through distinct 
    uncovered vertices $w_0$ and $w_1$ respectively at level 
    $i+1$ such that for some $j\in \{1,2,\dots,q\}$, $
    2d^2+4d\leq w_1(j)<w_0(j)\leq (i-1)d$. Then, by Lemma 
    \ref{lem:LinkLemma}, the link $w_0 \overset{u_1}{\longleftrightarrow} w_1$ extends to a distinguished $x,x'$ 
    chain of length 2 
 \[w_0 \overset{u_1}{\longleftrightarrow} w_1\overset{DSV(w_1,j)}{\longleftrightarrow}w_2\]
 for which there exists an $m_1$ such that $v_{i+1}(T^{m_1}x')=w_1$, $v_{i+1}(T^{m_1}x)=w_2$, and for all $m\in [0,m_1]$, we have $v_{i+1}(T^mx')=w_1$ and $v_{i+1}(T^mx)\in \{w_0,w_2\}$.
    
 Since $DSV(w_1,j)\in S(w_2)$, by Lemma 
 \ref{lem:BoundedMovement} and {Remark \ref{rem:MaxCoordinate}}, $1\leq w_1(j)-w_2(j)\leq d$. Together 
 with the fact that $2d^2+4d\leq w_1(j)\leq (i-
 1)d $, we have that $2d<2d^2+3d\leq w_2(j)\leq (i-1)d$. Hence, 
 by Lemma \ref{lem:LinkLemma}, we can extend the chain again. 
 Continue in this manner to produce a distinguished $x,x'$ chain:
   \[ w_0 \overset{u_1}{\longleftrightarrow} w_1 \overset{DSV(w_1,j)}
   {\longleftrightarrow} w_2{\longleftrightarrow}\dots 
   {\longleftrightarrow}w_{2d+3}.\]
Specifically, given $1<\ell\leq 2d+2$ and a $w_\ell$, we 
have $2d\leq 2d^2+4d-(\ell -1)d\leq w_{\ell}(j)\leq (i-1)d$. Then use 
the Link Lemma to extend the chain to $w_{\ell+1}$, where for $\ell >1$, $u_{\ell+1}=DSV(w_{\ell},j)$ and by Lemma \ref{lem:BoundedMovement}, 
$w_0(j)>w_1(j)>\dots>w_{2d+3}(j)$. 
This means all the splitting vertices are distinct, and by Remark 
\ref{rem:uncovered}, uncovered.

 By construction of the distinguished chain via the Link Lemma (Lemma \ref{lem:LinkLemma}), the orbit of $x$ only passes {through even}
 numbered splitting vertices and the orbit of $x'$ only passes {through odd} numbered splitting vertices, while moving through the chain.  Further, 
 since the $j'$th component of the splitting vertices are strictly 
 decreasing, for all $\ell\geq 2$, $w_{\ell}(j)-w_{\ell+1}(j)\geq 1$ for all $\ell \in \{1,\dots,2d+2\}$. So the $j$'th components of odd vertices differ by at least 2. Since the orbit of $x'$ passes {through only odd numbered} vertices, then for each $l\in\{0,1,\dots,d+1\}$, there is no $m\in [0,m_{2d+2}]$ such 
 that $v_{i+1}(T^mx')(j)=w_{2l+1}(j)-1$ or $w_{2l+1}(j)+1$. 
 See Figure \ref{fig:FullChain}.

  \begin{figure}
    \centering
    \begin{tikzpicture}
        \foreach \k in {0,2,4,6}{
        \fill(\k,0) circle (2pt);}
        \foreach \k in {1,3,5}{
        \fill(\k,1) circle (2pt);
        \draw(\k,1)--(-1+\k,0);
        \draw(\k,1)--(1+\k,0);}
        \node[below] at (0,0){$w_0$};
        \node[below] at (2,0){$w_1$};
        \node[below] at (4,0){$w_2$};
        \node[below] at (6,0){$w_3$};
        \node at (7,.5){$\dots$};
        \node[above] at (1,1){$u_0$};
        \node[above] at (3,1){$u_1$};
        \node[above] at (5,1){$u_2$};
        \node[left] at (.5,.5){\tiny $x$};
        \node[fill=white] at (1.5,.6){\tiny $x'$};
        \node[fill=white]  at (2.5,.4){\tiny $T^{m_1}x'$};
        \node[fill=white]  at (3.5,.6){\tiny $T^{m_1}x$};
         \node[fill=white] at (4.5,.4){\tiny $T^{m_2}x$};
         \node[fill=white]  at (5.5,.6){\tiny $T^{m_2}x'$};
 \foreach \k in {8,10,12}{
        \fill(\k,0) circle (2pt);}
        \foreach \k in {9,11}{
        \fill(\k,1) circle (2pt);
        \draw(\k,1)--(-1+\k,0);
        \draw(\k,1)--(1+\k,0);}
        \node[below] at (8,0){$w_{2d+1}$};
        \node[below] at (10,0){$w_{2d+2}$};
        \node[below] at (12,0){$w_{2d+3}$};
        \node[above] at (9,1){$u_{2d+1}$};
        \node[above] at (11,1){$u_{2d+2}$};
        \node[fill=white]  at (8.25,.35){\tiny $T^{m_{2d+1}}x'$};
        \node[fill=white]  at (9.5,.65){\tiny $T^{m_{2d+1}}x$};
         \node[fill=white] at (10.75,.35){\tiny $T^{m_{2d+2}}x$};
         \node[fill=white]  at (12,.65){\tiny $T^{m_{2d+2}}x'$};        
           \end{tikzpicture}
    \caption{For $l\geq 1$, $u_l=DSV(w_l,j)$.}
    \label{fig:FullChain}
 \end{figure}
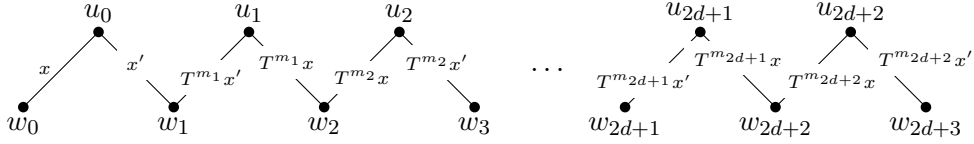

 To understand the switching of vertices at level $i+1$ of the two orbits, it is necessary to consider vertices at levels lower than $i+1$. In fact, we shall see that we do not have to look any lower than level $i+2$. 
 We will show that while the chain is being traversed the orbit of $x'$ hits a vertex z at level i+2 such that it cannot hit all vertices in S(z).\\

 Let $v_{i+2}(x')=z'$. Then $w_1\in S(z')$. Recall that for 
 $l\geq 2$,  $w_{1}(j)-w_l(j)\geq l-1$. 
 Since the $j$'th coordinates of vertices in the same source set differ by at most $d$ (Remark 
 \ref{rem:MaxCoordinate}), for $\ell \geq d+2$, $w_l\notin S(z')$. 
 Therefore there is an $m\in (0,m_{d+2}]$ such that $T^mx'$ is minimal into 
 level $i+2$ and $v_{i+2}(T^mx')=z\neq z'$. Further, $v_{i+1}(T^mx')$ 
 is some odd numbered splitting vertex $w_t$ with $t\leq d+2$. 
  By the way we have constructed our chain, we have that $2d\leq w_{t}(j)\leq (i-1)d$. Since $w_t\in S(z)$, it follows that $2d\leq z(j)\leq id$. 

Let $r=w_t(j)$, where $w_t$ is as above, so that $z(j) \geq r$.
 In the case $z(j)=r$, by Lemma \ref{lem:SourceSet}, we would have a vertex $v\in S(z)$ such that $v(j)=z(j)-1=r-1$. 
 In the case that $r+1\leq z(j)\leq r+d$, $z(j)-(r+1)$ is in the set $\{0,\dots, d-1\}$. Then by Lemma \ref{lem:SourceSet}, there is a vertex $v\in S(z)$ such that $v(j)=z(j)-(z(j)-(r+1))=r+1$. 
 Combining these two cases we have that there is a vertex $v\in S(z)$ 
 such that either $v(j)=r-1$ or $v(j)=r+1$.
 
 Notice that since $t\leq d+2$,
 $w_{t}(j)-w_{2d+3}(j)\geq w_{d+2}(j)-w_{2d+3}(j)\geq 2d+3-(d+2)=d+1$. 
 Hence (again by Remark \ref{rem:MaxCoordinate}) 
 $v_{i+1}(T^{m_{2d+2}}x')=w_{2d+3}\notin S(z)$. See Figure \ref{fig:LiP2}.
 
 \begin{figure}
     \centering
    \begin{tikzpicture}
    \draw(0,1)--(1,2)--(2,1)--(3,0);
    \draw (5,2)--(6,1)--(6,0)--(7.5,1);
    \draw (4.5,1)--(6,0);
    \draw (10,2)--(11,1)--(10,0);
    \fill (0,1) circle (2pt);
    \fill (1,2) circle (2pt);
    \fill (2,1) circle (2pt);
    \fill (3,0) circle (2pt);
    \fill (4.5,1) circle (2pt);
    \fill (5,2) circle (2pt);
    \fill (6,1) circle (2pt);
    \fill (7.5,1) circle (2pt);
    \fill (6,0) circle (2pt);
    \fill (11,1) circle (2pt);
    \fill (10,2) circle (2pt);
    \fill (10,0) circle (2pt);
    \node[below] at (0,1){$w_0$};
    \node[below left] at (2,1){$w_1$};
    \node[below] at (3,0){$z'$};
    \node[right] at (6,1){$w_l$};
    \node[below] at (6,0){$z$};
    \node[right] at (11,1){$w_{2d+3}$};
    \node[below] at (10,0){$z''\neq z$};
    \node[fill=white]  at (2.5,.5){\tiny $x'$};
    \node[fill=white]  at (6,.6){\tiny $T^mx'$};
    \node[fill=white]  at (10.5,.5){\tiny $T^{m_{2d+2}}x'$};
        \node[fill=white]  at (1.5,1.5){\tiny $x'$};
    \node[fill=white]  at (5.5,1.6){\tiny $T^mx'$};
    \node[fill=white]  at (10.5,1.5){\tiny $T^{m_{2d+2}}x'$};
    \node at (3,1){$\dots$};
    \node at (8.5,1){$\dots$};
    \end{tikzpicture}
     \caption{The orbit of $x'$ must pass through at least 3 distinct vertices at level $i+2$ between time $0$ and $m_{2d+2}$. So it must pass through all the vertices in $S(z)$ between time $1$ and $m_{2d+2}$.}
     \label{fig:LiP2}
 \end{figure}
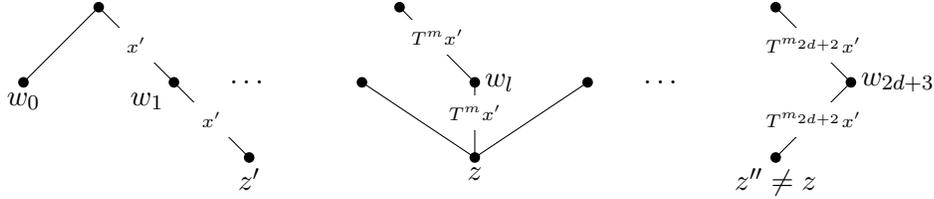
 
 By the 
 definition of the Vershik map, between time $m$ and time $m_{2d+2}$ 
 the orbit of $x'$ must pass through all of the vertices in $S(z)$. 
 However, this contradicts the fact that there is no $m'\in [0,m_{2d+2}]$ 
 such that $v_{i+1}(T^{m'}x')(j)=r-1$ or $r+1$. Therefore, there can be no depth $i$ pair of paths, each with a dense infinite orbit. 
\end{proof}

{\begin{remark}
    {The constructions and arguments presented above might extend to more general systems, beyond polynomial shape. 
    	As a first step, one could consider systems defined by different polynomials at different levels, being careful about the degree and the number of variables at each level. 
    	Even more generally, the properties of polynomial shape systems that allow for the construction of chains can be abstracted to define a class of diagrams, 
    	free of coordinate representations for vertices and all the accompanying structure, for which one might still be able to prove inherent expansiveness. 
    	Complete details have not yet been established.}
    	    \end{remark}}

\end{document}